\documentclass{amsart}

\usepackage[dvipsnames]{xcolor}
\usepackage[english]{babel}
\usepackage{amsfonts}
\usepackage{amsthm}
\usepackage{amsbsy}
\usepackage{amssymb}
\usepackage{graphicx}
\usepackage{eso-pic}
\usepackage{tikz}
\usepackage{xfrac}
\usepackage{float}
\usepackage[bookmarks=true,colorlinks=true,urlcolor=black,linkcolor=black,citecolor=black,pagebackref=false]
{hyperref}  
\usepackage{bookmark}
\usepackage{amsmath}
\usepackage{subfigure}
\usepackage{hyperref}
\usepackage{scalerel,stackengine}
\usepackage{epigraph}
\usepackage{placeins}
\usetikzlibrary{matrix,arrows}
\usepackage[margin=3cm]{geometry}
\usepackage{enumitem}

\usepackage{mathtools}
\usepackage{hhline}

\usepackage{amsmath,amssymb,tikz-cd}
\usepackage[inline]{asymptote}
\usepackage{comment}

\makeatletter
\newcommand{\mylabel}[2]{#2\def\@currentlabel{#2}\label{#1}}
\makeatother
\usepackage{todonotes}

\newtheorem{lemma}{Lemma}[section]
\newtheorem{thm}[lemma]{Theorem}
\newtheorem{prop}[lemma]{Proposition}
\newtheorem{cor}[lemma]{Corollary}
\newtheorem*{cor*}{Corollary}

\theoremstyle{definition}
\newtheorem{defn}[lemma]{Definition}
\newtheorem{notat}[lemma]{Notation}

\newtheorem{quest}[lemma]{Question}

\theoremstyle{remark}
\newtheorem{rem}[lemma]{Remark}

\newcommand\restr[2]{{
  \left.\kern-\nulldelimiterspace 
  #1 
  \littletaller 
  \right|_{#2} 
  }}

\newcommand{\littletaller}{\mathchoice{\vphantom{\big|}}{}{}{}}

\newcommand{\lk}{L}

\newcommand{\matN}{\ensuremath {\mathbb{N}}}
\newcommand{\matR} {\ensuremath {\mathbb{R}}}

\newcommand{\matH} {\ensuremath {\mathbb{H}}}
\newcommand{\matS} {\ensuremath {\mathbb{S}}}

\newcommand{\Isom}{\ensuremath {\mathrm{Isom}}}

\author{Edoardo Rizzi}
\address{Scuola Normale Superiore \newline Piazza dei Cavalieri 7, 56126 Pisa, Italy}
\email{edoardo.rizzi@sns.it}

\thanks{}

\title{Some cusp-transitive hyperbolic 4-manifolds}

\begin{document}

\begin{abstract}
We realize 4 of the 6 closed orientable flat 3-manifolds as a cusp section of an orientable finite-volume hyperbolic 4
-manifold whose symmetry group acts transitively on the set of cusps.
\end{abstract}

\maketitle

\section{Introduction}

A (complete, finite-volume)
hyperbolic manifold is \emph{cusp transitive} if its isometry group acts transitively on the set of cusps. Very special cases are the \emph{$1$-cusped} manifolds (i.e.~manifolds with a single cusp), of which infinitely many examples are well known in dimension $2$ and $3$
, and one has been exhibited in dimension 4 for the first time in $2013$ \cite{KM}. 
In higher dimension $n > 4$ we do not know whether there exists a $1$-cusped hyperbolic $n$-manifold. In fact the existence of $1$-cusped manifolds is highly non-trivial, for example there is no $1$-cusped arithmetic orbifold of dimension $n\ge 30$ \cite{MR3070517}. Concerning cusp-transitive manifolds of dimension $n>4$, we are not aware of explicit examples in the literature. By mirroring some well-known right angled polytopes, it is easy to obtain some examples of dimension up to $8$ with toric cusps.

The \emph{type} of a cusp of a hyperbolic manifold is the diffeomorphism class of its section, which is a flat closed hypersurface. Each flat closed $n$-manifold is realized as a cusp type of some hyperbolic $(n+1)$-manifold \cite{MR2525031} (see also \cite{N,LR2,M}). The latter manifold has generally several other cusps, whose type does not appear controllable with the separability methods of \cite{N,LR2,M,MR2525031}. In the orientable setting, there are obstructions for a closed flat $(4n-1)$-manifold to be the cusp type of a $1$-cusped $4n$-manifold \cite{LR}. In this paper we are interested in which closed flat manifold can be realized as the cusp type of a cusp-transitive hyperbolic manifold.

For reasons that will become clear later, this article deals with $4$-dimensional hyperbolic manifolds. Recall that there are precisely $6$ closed orientable flat $3$-manifolds up to diffeomorphism: 
$E_1, \ldots, E_6$. Here $E_i$ is a mapping torus over $S^1 \times S^1$ with monodromy of order $1, 2, 3, 4, 6$ for $i = 1, \ldots, 5$, respectively, and $E_6$ is a rational homology sphere \cite[Section 12.3]{MA}. The manifold $E_1$ is the $3$-torus. We call $E_2$ (resp.~$E_3, E_4, E_5$) the \emph{$\frac{1}{2}$-twist} (resp.~\emph{$\frac{1}{3}$-twist}, \emph{$\frac{1}{4}$-twist}, \emph{$\frac{1}{6}$-twist}) \emph{manifold}, while $E_6$ is the so-called \emph{Hantzsche-Wendt manifold}.

There exist $1$-cusped orientable $4$-manifolds with cusp type $E_1$ \cite{KM} and $E_2$ \cite{KS}, while there is no $1$-cusped orientable $4$-manifold with cusp type $E_3$ or $E_5$ \cite{LR}. We refer to the discussion in  \cite[Sections 2.5 and 2.6]{MR3888615}, for this and related issues in dimension four. Moreover, there exists a cusp-transitive $4$-manifold with cusp type $E_6$ \cite{FKS}. We reprove this fact here. 
The novelty of this paper is the existence of a cusp-transitive hyperbolic $4$-manifold with cusp type the $\frac{1}{4}$-twist manifold $E_4$. Our construction actually realizes more cusp types. Specifically we prove:

%

\begin{thm} \label{thm:main}
For each $i = 1, 2, 4, 6$ there exists a cusp-transitive orientable hyperbolic $4$-manifold $M_i$ with cusps of type $E_i$.
\end{thm}



Our method to produce cusp-transitive hyperbolic manifolds a priori works in arbitrary dimension (but not a posteriori: there is no finite-volume hyperbolic Coxeter polytope of dimension $\geq 996$ \cite{P}). The manifold $M_i$ is built by orbifold covering a hyperbolic Coxeter polytope $P_0$ such that:

\begin{itemize}
\item[\mylabel{item:a}{(a)}] it has exactly one ideal vertex;
\item[\mylabel{item:b}{(b)}] 
if a bounded facet and an unbounded facet intersect, then 
their dihedral angle is 
an even submultiple of $\pi$.
\end{itemize}

The construction roughly goes as follows. We glue toghether some copies of $P_0$, so as to get a hyperbolic manifold with corners $R_i$ satisfying the following properties. First, $R_i$ is $1$-cusped, and this will follow from \ref{item:a}. By construction the cusp will have section $E_i$ ($i=1,2,4,6$). Second, $R_i$ is locally a Coxeter polytope (a so-called \emph{reflectofold}), and this will follow from \ref{item:b}: when gluing two facets the dihedral angle is indeed doubled, and hence it is still an integral submultiple of $\pi$. It is homeomorphic to $E_i \times [0,+\infty)$, and its boundary is stratified into connected closed sets: facets, corners, edges and vertices of dimension $3, 2, 1$ and $0$, respectively. Third, we need to perform the gluing in such a way that $R_i$ is \emph{developable}, that is (see for instance \cite[Section 3]{CD}):
\begin{enumerate}
\item \label{item:1} the facets are embedded (and not just immersed) hyperbolic manifolds with corners;
\item \label{item:2} if two facets intersect, then the dihedral angles at all the corresponding corners coincide.
\end{enumerate}
These properties of $R_i$ allow us to apply to $R_i$ Davis' ``basic construction'' \cite{D}, and to get a manifold $M_i$ tessellated by some copies of $R_i$, with a group of symmetries $G_i$ such that $M_i/G_i \cong R_i$. So $M_i$ is cusp transitive and its cusps have section $E_i$.

Ensuring \eqref{item:1} and \eqref{item:2} is the most technical point of the construction. Indeed, there is an easy way to glue some copies of $P_0$ in order to get a $1$-cusped reflectofold $R_i$ with cusp section $E_i$, but the resulting $R_i$ would not be developable. Hence we iteratively double the polytope $P_0$, obtaining a sequence $P_0,\dots,P_m$ of polytopes which satisfy the properties \ref{item:a} and \ref{item:b}. We continue to double until we find a gluing for a polytope $P_m$ giving a developable $R_i$.

We want to obtain a cover $M_i$ of $R_i$, with cusps isometric to the one of $R_i$, and this will follow from the fact that $M_i$ is tessellated by copies of $R_i$. Note that a generic cover of $R_i$ has cusps non-homeomorphic to the one of $R_i$. The authors in \cite{N,MR2525031} find an orbifold with one cusp of the desired type, and then, by a separability argument, they find a manifold cover with a cusp of the same type. We do the same thing, but our construction guarantees that all the cusps of $M_i$ are of the same type of the cusp of $R_i$, since the construction is more geometric. Indeed we obtain $M_i$ by gluing copies of $R_i$, which is built explicitly. 



Among the hyperbolic Coxeter $n$-polytopes with $n \geq 4$ that we have in hand,\footnote{See Felikson's web page \href{www.maths.dur.ac.uk/users/anna.felikson/Polytopes/polytopes.html}{\url{www.maths.dur.ac.uk/users/anna.felikson/Polytopes/polytopes.html}}.} we found only one $P_0$ satisfying \ref{item:a} and \ref{item:b}, among Im Hof's polytopes associated to Napier's cycles \cite{IH}. We find cusp-transitive manifolds of dimension four because $P_0$ is $4$-dimensional, and we realize only some cusp types because of the particular link type of the ideal vertex of $P_0$: a prism over a $(2,4,4)$-triangle.  Note indeed that $E_1, E_2, E_4$ and $E_6$ can be tessellated by right parallelepipeds, and thus by such a prism. We would like to apply our construction to other polytopes, for example in dimension greater than $4$, but we did not find any other polytope with the desired properties.

\begin{quest}
Does there exist a finite-volume hyperbolic Coxeter polytope of dimension $n \geq 4$ satisfying \ref{item:a} and \ref{item:b}? If the dimension is $n=4$ we require that the link of the ideal vertex is not a parallelepiped nor a prism over a $(2,4,4)$-triangle.
\end{quest}

A positive answer to the latter question may give, by our methods, a positive answer to the following question.

\begin{quest}
Does there exist a cusp-transitive hyperbolic $4$-manifold with cusps of type $E_3$ or $E_5$?
\end{quest}

We would like to improve the method in a future work, with the hope of producing original examples of $1$-cusped manifolds. In principle this may be done, instead of developing such an $R_i$, by closing it up gluing its facets. This is much more difficult (and sometimes impossible by some immediate obstructions), but has the advantage that more polytopes may be used, since in this case the quite restrictive property \ref{item:b} is not necessarily needed. 

\begin{quest}
For which $i = 3, 4, 5, 6$ does there exist a $1$-cusped hyperbolic $4$-manifold whose cusp has type $E_i$? Can moreover such a $4$-manifold be orientable when $i = 4, 6$?
\end{quest}

\begin{quest}
For which dimension $n>4$ does there exist a $1$-cusped hyperbolic $n$-manifold?
\end{quest}

The paper is organized as follows. In Section \ref{section2} we describe how to obtain a cusp-transitive manifold from a $1$-cusped developable reflectofold. In Section \ref{sec:proof} we double the polytope $P_0$ many times in order to obtain two Coxeter polytopes which we will use in Section \ref{reflect} in order to get some $1$-cusped developable reflectofolds.

I would like to thank my advisor Stefano Riolo for all his help.

%
%
%

\section{Cusp-transitive manifolds from \texorpdfstring{$1$}{1}-cusped reflectofolds} \label{section2}

In this section we describe a general method to build a cusp-transitive, hyperbolic manifold from a $1$-cusped reflectofold.

Recall that a \emph{hyperbolic Coxeter polytope} is a finite convex polytope $P \subset \matH^n$ whose dihedral angles are integral submultiples of $\pi$. We call \emph{facets} and \emph{ridges} the $(n-1)$-dimensional and $(n-2)$-dimensional faces of such a $P$, respectively. We refer to Vinberg's paper \cite{V} for the general theory of hyperbolic Coxeter polytopes and groups.

To a hyperbolic Coxeter polytope $P$ one associates a decorated graph, called the \emph{Coxeter diagram} of $P$ and defined as follows. The graph has a node for each bounding hyperplane, and an edge joining nodes $i$ and $j$ has label $m_{ij}$ if the corresponding hyperplanes intersect with dihedral angle $\frac{\pi}{m_{ij}}$. By usual convention, the label is $m_{ij} = \infty$ when the two hyperplanes are tangent at infinity, and the edge of the graph is omitted (resp.~dashed) when $m_{ij}=2$ (resp.~the two hyperplanes are ultraparallel). 

\begin{defn}
    We say that a complete hyperbolic manifold $R$ with boundary is a \emph{reflectofold} if $R$ is locally a hyperbolic Coxeter polytope.\footnote{Though we will not strictly need to deal with the orbifold theory, let us notice that a reflectofold $R$ is isometric to a hyperbolic orbifold (sometimes called in the literature \emph{Coxeter orbifold}). In other words, we have $R \cong \matH^n/\Gamma$ for some discrete subgroup $\Gamma < \Isom(\matH^n)$. To avoid confusion with the terminology, let us notice the following.
    Even if in the category of manifolds with boundary sometimes $\partial R \neq \emptyset$ and $R$ is orientable, if seen in the orbifold category such an $R$ is non-orientable and without boundary. Unless otherwise stated, we will consider $R$ as a manifold with boundary.}    
\end{defn}

Let $R$ be a reflectofold. The stratification of each local model $P$ of $R$ into $k$-dimensional faces, $k = 0, \ldots, n$, naturally induces a stratification of $R$ into maximal, connected, totally geodesic submanifolds (with boundary), called $k$-\emph{faces}. The $(n-1)$-faces and the $(n-2)$-faces of $R$ will be called \emph{facets} and \emph{corners}, respectively. The \emph{dihedral angle} of a corner is the dihedral angle of the corresponding ridge of a local model. 

\begin{defn}
   A reflectofold is \emph{developable} if the following hold:
    \begin{itemize}
        \item[\mylabel{EF}{(EF)}] \emph{Embedded faces}: For each corner $C$ there are two distinct facets $F$ and $F'$ such that $C \subset F \cap F'$.
        \item[\mylabel{AC}{(AC)}] \emph{Angle consistency}: If two distinct facets $F$ and $F'$ intersect, then the dihedral angles of all the corners in $F \cap F'$ coincide.
    \end{itemize}    
\end{defn}


Given a developable reflectofold $R$, we denote by $G_R$ the Coxeter group defined by the following presentation. For each facet $f$ of $R$ there is the generator $f$ and the relator $f^2$. Moreover, there is the relator $(fg)^k$ for every pair of facets $f$ and $g$ which intersect with dihedral angle $\frac{\pi}{k}$.


Let us now apply Davis' ``basic construction'' to $R$ and $G_R$ \cite{D}.
We define a space $\widetilde{R}$ as follows. We take $\{gR\}_{g\in G_R}$, a set copies of $R$. For every generator $f$ of $G_R$, we glue the copies $gR$ and $fgR$ identifying the two facets corresponding to $f$ via the map induced by the identity.

\begin{prop}
\label{tassell}
    Let $R$ be a developable reflectofold. Then 
    $R$ is isometric to the quotient of a hyperbolic manifold $M$ tessellated by copies of $R$, by a finite group $G$ of isometries. 
\end{prop}

By \emph{tessellated by copies of $R$} we mean that $M$ can be decomposed into some copies of $R$ in such a way that the intersection of any two copies is a union of faces.

\begin{proof}
    
    
    We begin proving that $\widetilde{R}$ is a hyperbolic manifold.

    Internally to the copies of $R$ the space $\widetilde{R}$ is locally isometric to $\matH^n$. We have to check what happens near the boundary of the copies of $R$. In particular, we have to check that, given a $k$-face $f$ of a copy $gR$, the link of $f$ in $\widetilde{R}$ is isometric to the round sphere $\matS^{n-k-1}$. 

    The link of a $k$-face $F$ of $R$ is a spherical Coxeter $(n-k-1)$-simplex $S$. 
    It is well known \cite[Section $4.1$]{D} that the abstract Coxeter group $G_S$ associated to $S$ embeds in $G_R$, it is generated by the corresponding subset of the generators of $G_R$ and the relators between them are the ones from the presentation of $G_R$.
    Hence the link of the $k$-face $F$ in $\widetilde{R}$ is the basic construction associated to $S$ and $G_S$, and is isometric to $\matS^{n-k-1}$. We have proved that $\widetilde{R}$ is a hyperbolic manifold. It is complete by construction.



    Since Coxeter groups are virtually torsion free \cite[Corollary $D.1.4$]{D}, we can take a normal subgroup $G'_R\triangleleft G_R$ of finite index and with no torsion.
    The group $G_R$ acts on $\widetilde{R}$ by isometry preserving the tassellation of $\widetilde{R}$ in copies of $R$, and $\widetilde{R}/G_R\cong R$. 
    Since $G'_R$ is torsion free it acts freely on $\widetilde{R}$.
    Hence $M\coloneqq \widetilde{R}/G'_R$ is a hyperbolic manifold. We set $G\coloneqq G_R/G'_R$.
    Since $\widetilde{R}/G_R \cong R$, we have $M/G \cong R$.   
    \end{proof}

Recall now the definition of cusp transitivity from the introduction. We immediately get:

\begin{cor}
\label{maincor}
    Let $R$ be an orientable, finite-volume, developable reflectofold. If $R$ has compact boundary and exactly one cusp $C$, then there exists an orientable, cusp-transitive, hyperbolic manifold $M$ with cusps isometric to $C$. 
\end{cor}
\begin{proof}
    The manifold $M$ of Proposition \ref{tassell} is cusp transitive  and its cusps are isometric to $C$ because $R$ is $1$-cusped and $M/G \cong R$. If $M$ is non-orientable, it can be replaced by its orientable double cover $\widetilde{M}$. Indeed, for every cusp $D$ of $M$, the cover $\widetilde{M}$ has two cusps isometric to $D$ (since $R$ is orientable). Moreover, $\widetilde{M}$ is cusp-transitive. Indeed, we can send every cusp to another one using the involution $i$ of $\widetilde{M}$ such that $\widetilde{M}/\langle i \rangle\cong M$ and the liftings of the isometries of $M$ which realize the cusp-transivity of $M$ (we can lift them since an isometry sends an orientable tubolar neighborhood of a loop to an orientable tubolar neighborhood of a loop).
\end{proof}



Hence, in order to prove Theorem \ref{thm:main}, we will build an orientable, finite-volume, $1$-cusped, developable reflectofold with compact boundary, whose cusp has type $E_i$, for $i=1,2,4,6$.

\section{The polytopes} \label{sec:proof}

In this section we build some Coxeter polytopes satisfying \ref{item:a} and \ref{item:b}. We will use them in Section \ref{reflect} to build some $1$-cusped developable reflectofolds.

In Section \ref{P0} we introduce Im Hof's Coxeter polytope $P_0$. Then, in Section \ref{P7eP8} we describe a way to obtain a sequence $P_0,P_1,\dots,P_8$ of Coxeter polytopes satisfying \ref{item:a} and \ref{item:b}, by iteratively doubling $P_0$, and we describe how to study them. In Section \ref{costruzione} we build the sequence of polytopes, and we obtain the information on the polytopes using the results of Section \ref{P7eP8}.

\subsection{The polytope \texorpdfstring{$P_0$}{P\_0}}\label{P0}
In this section we introduce a Coxeter polytope from \cite{IH}.

Consider the following Coxeter diagram $D$:

\begin{center}
\begin{tikzpicture}[thick,scale=0.8, every node/.style={transform shape}]
\node[draw, circle, inner sep=2pt, minimum size=3mm] (1) at (-1.56,-1.25)  {6};
\node[draw, circle, inner sep=2pt, minimum size=3mm] (2) at (0,-2)         {5};
\node[draw, circle, inner sep=2pt, minimum size=3mm] (3) at (1.56,-1.25)   {4};
\node[draw, circle, inner sep=2pt, minimum size=3mm] (4) at (1.95,0.45)  {3};
\node[draw, circle, inner sep=2pt, minimum size=3mm] (5) at (0.87,1.80)  {2};
\node[draw, circle, inner sep=2pt, minimum size=3mm] (6) at (-0.87,1.80) {1};
\node[draw, circle, inner sep=2pt, minimum size=3mm] (7) at (-1.95,0.45) {7};

\node at (-0.87,-1.80)  {$4$};
\node at (0.87,-1.80)   {$4$};
\node at (1.95,-0.45)   {$6$};
\node at (1.56,1.25)  {$4$};
\node at (0,2)        {$\infty$};
\node at (-1.56,1.25) {$4$};
\node at (-1.95,-0.45)  {$6$};

\draw (1) -- (2) node[above,midway]   {};
\draw (2) -- (3) node[above,near end] {};
\draw (3) -- (4) node[above,midway]   {};
\draw (4) -- (5) node[above,midway]   {};
\draw (5) -- (6) node[above,midway]   {};
\draw (6) -- (7) node[above,midway]   {};
\draw (7) -- (1) node[above,midway]   {};
\end{tikzpicture}
\end{center}

\begin{prop}
\label{riolo}
The graph $D$ above is the Coxeter diagram of a finite-volume hyperbolic Coxeter $4$-polytope $P_0$ which satisfies \emph{\ref{item:a}} and \emph{\ref{item:b}}. The horospherical link of the unique ideal vertex of $P_0$ is a Euclidean right prism over a triangle with inner angles $\frac{\pi}{2}$, $\frac{\pi}{4}$, $\frac{\pi}{4}$, and its Coxeter diagram is the subdiagram of $D$ spanned by the vertices $1,2,4,5,6$.
\end{prop}
\begin{proof}
    We already know from \cite{IH} that $P_0$ has finite volume. By \cite{V}, the ideal vertices correspond to the maximal affine subdiagrams of the Coxeter diagram. In $D$ we have exactly one of this kind (see \cite[Table 2]{V}), spanned by the vertices $1,2,4,5,6$.
\end{proof}


\subsection{The sequence of polytopes}\label{P7eP8} The purpose of this section is to fix some notation, and to describe how to build and study some new Coxeter polytopes satisfying \ref{item:a} and \ref{item:b} by doubling iteratively $P_0$ along some facets.




\begin{defn}
We say that a facet $F$ of a polyotpe $P$ is \emph{admissible} if, whenever it intersects another facet $K$ of $P$, then the dihedral angle between $F$ and $K$ is equal to $\frac{\pi}{2k}$ for some $k\in\matN$.
\end{defn}

We notice that every facet of $P_0$ is admissible. Indeed the numbers labelling the Coxeter diagram $D$ are even.

Given a  hyperbolic $n$-polytope $P$ and a facet $F$ of $P$, we denote by $r_F\colon\matH^n\rightarrow \matH^n$ the reflection through the unique hyperplane that contains $F$.
In Section \ref{costruzione}, we will construct a sequence of Coxeter polytopes in $\matH^n$ satisfying \ref{item:a} and \ref{item:b}:
$$
P_0,P_1,P_2,P_3,P_4,P_5,P_6,P_7,P_8,
$$
where $P_{n+1}=P_n\cup r_{F_n}(P_n)$, for some admissible, non-compact facet $F_n$ of $P_n$.
We say that $P_{n+1}$ is the \emph{double of $P_n$ along $F_n$}. Before the actual definition of $P_n$, we now fix some notation and deduce some information on such a sequence of polytopes in general. 

\begin{rem}
    Since $P_0$ is a Coxeter polytope and the facet $F_n$ of $P_n$ will be chosen to be admissible, also $P_1,\dots,P_8$ will be Coxeter polytopes.
\end{rem}

Let $V$ be the only ideal vertex of $P_0$ (recall Proposition \ref{riolo}). Notice that $P_n$ has exactly one ideal vertex for all $n$, and it is always $V$. Indeed, we always double along a non-compact facet.

Let $\lk_n$ be the link of the ideal vertex $V$ of $P_n$. It is a $3$-dimensional Euclidean polytope well-defined up to scaling.


\begin{rem}
    There is a natural bijection between the set of non-compact facets of $P_n$ and the set of the facets of $\lk_n$. Indeed, if we take a ``small" orosphere $O$ centred at $V$, then $O\cap P_n$ can be identified to $\lk_n$ and every facet of $\lk_n$ can be identified with the intersection of $O$ with a non-compact facet of $P_n$. Vice versa, every non-compact facet $F$ of $P_n$ meets $O$, and $O\cap F$ is a facet of $\lk_n$. Indeed, $P_n$ has exactly one vertex at infinity.
\end{rem}

\begin{notat}
We will call the facets of $\lk_0$ with the same name of the facets of $P_0$.
\end{notat}




Note that $\lk_{n+1}$ is the double of $\lk_n$ along its facet $F_n$.

The construction of $P_n$ induces a tessellation of $P_n$ in copies of $P_0$. In particular, we also have a tessellation of the facets of $P_n$ in copies of facets of $P_0$. We say that a facet is \emph{of type} $i$ if it is tessellated into copies of the facet $\textit{\textbf{i}}$ of $P_0$.

\begin{defn}
    Let $A$ be a facet of $P_n$. Let $A_1,\dots,A_k$ be the facets that meet $A$ and $\alpha_i$ be the dihedral angle at $A\cap A_i$. We define $I_n(A):=\{(A_1,\alpha_1),(A_2,\alpha_2),\dots,(A_k,\alpha_k)\}$. 
    
\end{defn}

\begin{rem}
\label{lista}
    If $(A,\frac{\pi}{2})\in I_n(F_n)$, then in $P_{n+1}=P_n\cup r_{F_n}(P_n)$ we have that $A\cup r_{F_n}(A)$ is a unique facet. Otherwise, if $(A,\frac{\pi}{2})\notin I_n(F_n)$ then $A$ and $r_{F_n}(A)$ are two distinct facets of $P_{n+1}$.
\end{rem}

\begin{notat}
    From now on, we will call a facet with the same name of the hyperplane that cointains it. 
    Hence, if $(A,\frac{\pi}{2})\in I_n(F_n)$, we have that $A\cup r_{F_n}(A)$ is a facet of $P_{n+1}$ that we call $A$ by a little abuse. 
\end{notat}

We now begin the first step of our construction.

\begin{defn}
    We define $P_1=P_0 \cup r_5(P_0)$.
\end{defn}

The facets of $P_1$ are: $\textbf{1}, \textbf{2},\textbf{3}, \textbf{4}, r_5(\textbf{4}), \textbf{6}, r_5(\textbf{6}), \textbf{7}$. Indeed, we can deduce the list using Remark \ref{lista} and the fact that $(\textbf{1},\frac{\pi}{2}),(\textbf{2},\frac{\pi}{2}),(\textbf{3},\frac{\pi}{2}),(\textbf{7},\frac{\pi}{2})\in I_0(\textbf{5})$, while $(\textbf{4},\frac{\pi}{2}),(\textbf{6},\frac{\pi}{2})\notin I_0(\textbf{5})$.

For a shorter notation we denote $r_5(\textbf{4})$ by $\textbf{4}_5$, and so on. Hence, with this convention, the facets are:
$\textbf{1}, \textbf{2}, \textbf{3}, \textbf{4}, \textbf{4}_5, \textbf{6}, \textbf{6}_5, \textbf{7}$. In this case the facets $\textbf{4}$ and $\textbf{4}_5$ are facets of type $4$, while $\textbf{6}$ and $\textbf{6}_5$ are facets of type $6$, and $\textbf{7}$ is a facet of type $7$.


By Proposition \ref{riolo} we know the Coxeter diagram for $\lk_0$.
Hence, the links $\lk_0$ and $\lk_1$ of the ideal vertex $V$ of $P_0$ and $P_1$ are the ones in Figure \ref{f1}.

\begin{figure}

\includegraphics{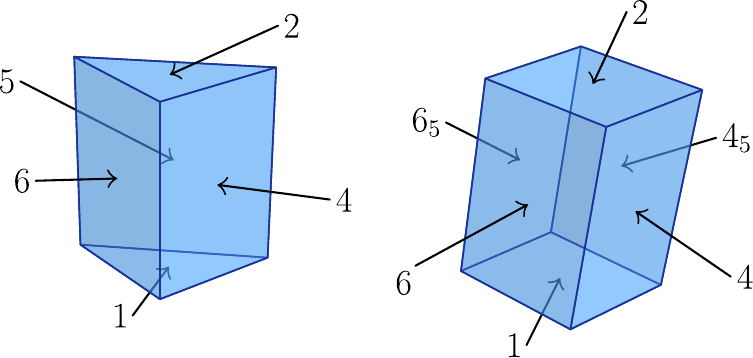}

    \caption{The link $L_0$ (left) and the link $L_1$ (right).}\label{f1}
\end{figure}

Since $P_1$ is tessellated by two copies of $P_0$, we have a tessellation of every facet of $P_1$ in one or two copies of a facet of $P_0$, and similarly for $L_1$ with $L_0$.

We collect some information on $P_0$ and $P_1$ via some pictures representing the facets of $L_0$ and $L_1$, respectively. The facets of $\lk_0$ and $\lk_1$ are represented in Figure \ref{f2}.
\begin{figure}
\includegraphics{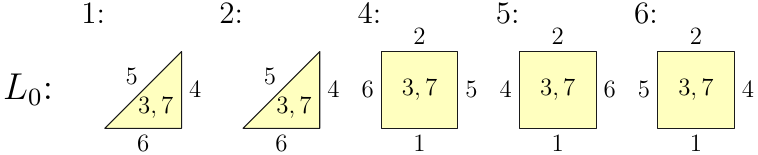}
\\
\vspace{5ex}
\includegraphics{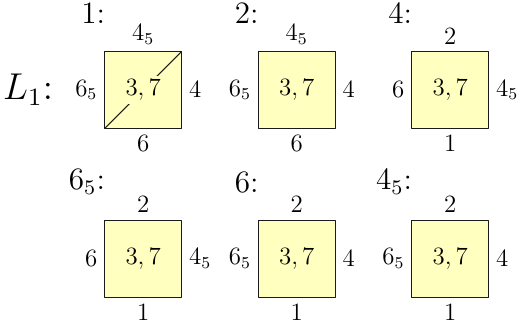}
    \caption{The facets of $L_0$ (top) and the facets of $L_1$ (bottom).}\label{f2}
\end{figure}
The meaning of these pictures is the following. Recall that $L_1$ is a right parallelepiped, so its facets are $6$ rectangles. Each of these rectangles is tessellated by one or two copies of a facet of $L_0$. In the picture, each of such rectangles is tiled by some tiles (squares or triangles). Each tile also corresponds to a tile of the tessellation of a facet of $P_1$. In the picture, each tile contains the labels of the compact facets of $P_1$ that intersect the corresponding tile in $P_1$. To avoid writing the same label in two adjacent tiles, we put the label on the edge dividing them, like for instance the labels $3$ and $7$ of the facet $\textbf{1}$. Moreover, outside of the tiles we have written the labels of some non-compact facets of $P_1$. The label of a non-compact facet $N$ is drawn near the edge of a tile if the corresponding tile in $P_1$ (a copy of a facet of $P_0$ in $P_1$) intersects $N$.

\begin{rem}
\label{motivazione}
    Since the link of the ideal vertex $V$ of $P_1$ is a parallelepiped, we could take the small covers of the cube \cite[Section 3]{FKS} to obtain three of the four desired reflectofolds (the ones with cusp section the $3$-torus, the $\frac{1}{2}$-twist manifold and the Hantzsche-Wendt manifold). The problem is that these reflectofolds are not developable. Hence we will iteratively double the polytope until we find a polytope $P$ such that we can glue $P$ in order to obtain a $1$-cusped, developable reflectofold with the desired cusp section.
\end{rem}

For the other steps of the construction we will keep track of the following information on $P_n$:
\begin{itemize}
    \item[\mylabel{item:I1}{(I1)}] the list of the facets;
    \item[\mylabel{item:I2}{(I2)}] the adjacency graphs of the facets of type $3$ and $7$;
    \item[\mylabel{item:I3}{(I3)}] the picture of the facets of $L_n$ tessellated and labelled with the previous convention.
\end{itemize}

\begin{notat}
     We extend the notation given for Step $1$ for the facets of $P_1$ to the facets of $P_n$. For example, we will see in Section \ref{costruzione} that $(r_4\circ r_{r_5(4)} \circ r_2) (\textbf{3})$ is a facet of $P_4$, and it will be denoted as $\textbf{3}_{4,4_5,2}$. The convention will be similar for the other facets.
\end{notat}

We will represent the adjacency graphs of the facets of type $3$ and of type $7$ of $P_n$ separately, as follows.
There is a vertex for each type-$3$ (respectively type-$7$) facet, and we connect two vertices with an edge with label $k$ if the two facets meet with dihedral angle $\frac{\pi}{k}$ (including the case with $k=2$). There is no edge joining two vertices of the graph if the two facets are at positive distance (they cannot be tangent at infinity since they are compact).

We will represent each adjacency graph via the associated adjacency matrix: in the entry corresponding to the vertices $A$ and $B$ we put $1$ if $A=B$, we put $0$ if there is no edge between them, and $k$ if there is an edge with label $k$ between them. For more clarity we omit the $0$ in the entries. 




\begin{prop}
\label{angolo}
    If two facets of type $i$ and $j$ of $P_n$ meet, with $i\ne j$, then the dihedral angle between them is the same of the one between the facets $\textbf{i}$ and $\textbf{j}$ of $P_0$. In particular the facets $\textbf{i}$ and $\textbf{j}$ of $P_0$ meet.
\end{prop}
\begin{proof}
    The polytope $P_n$ is tessellated by some copies of $P_0$ and a facet of type $k$ is tessellated by some copies of the facet $\textit{\textbf{k}}$ of $P_0$. Hence the facet $A$ of type $i$ and the facet $B$ of type $j$ of $P_n$ meet in a copy of $P_0$. We deduce that the facets $\textit{\textbf{i}}$ and $\textit{\textbf{j}}$ of $P_0$ meet and the dihedral angle between them is the same of the dihedral angle between $A$ and $B$. 
\end{proof}

\begin{rem}
\label{cubo}
    If $n\ge 1$, then the link $L_n$ is a right parallelepiped. Indeed, $\lk_1$ is a right parallelepiped, at every step we double $P_n$ along a non-compact facet, and $P_n$ has exactly one ideal vertex. In particular, for every couple of non-compact facets of $P_n$ that meet, the corresponding dihedral angle is $\frac{\pi}{2}$.
\end{rem}

\begin{cor}
\label{suit}
    If $n\ge 1$, every facet of $P_n$ that is not of type $3$ or $7$ is non-compact and admissible.
\end{cor}
\begin{proof}
    Since $3$ and $7$ are the only compact facets of $P_0$, every facet of a different type from $3$ and $7$ in $P_n$ is non-compact.
    
    We show that every non-compact facet is admissible. Let $A$ be a non-compact type-$i$ facet of $P_n$. Let $B$ be another facet of $P_n$ that meets $A$. If $B$ is non-compact, then by Remark \ref{cubo} the dihedral angle between them is $\frac{\pi}{2}$. If $B$ is compact, then $A$ and $B$ have different type. Hence, by Proposition \ref{angolo} the dihedral angle between them is $\frac{\pi}{2k}$ for some $k$, since this is true for every couple of facets of $P_0$ that meet.
\end{proof}

We now state a proposition that will allow to recover the needed information \ref{item:I1},  \ref{item:I2}, \ref{item:I3} on $P_{n+1}$ starting from that of $P_n$, for $n\ge 1$.

\begin{notat}
    In the following, if two vertices $F$ and $G$ of a graph are joined by an edge with label $k$, we denote this edge by $(F, G; k)$.
\end{notat}

Recall that $P_{n+1}$  will be the double of $P_n$ along a non-compact, admissible facet called $F_n$.

\begin{prop}
\label{procedura}
   If $n\ge1$, the information \ref{item:I1},  \ref{item:I2}, \ref{item:I3} on $P_{n+1}$ is obtained from the one on $P_n$ as follows.
    \begin{enumerate}
        \item[\ref{item:I1}] For every facet $G\ne F_n$ in the list \ref{item:I1} of $P_n$:
        \begin{itemize}
            \item If the label $G$  is in the picture of $F_n$ in \ref{item:I3} of $P_n$:
            \begin{itemize}
                \item if $G$ is of the same type as $F_n$, then add $G$ to the list \ref{item:I1} of $P_{n+1}$;
                \item if $G$ and $F_n$ are respectively of type $i$ and $j$ with $i\ne j$ and in the Coxeter diagram of $P_0$ there is not an edge between the vertices $i$ and $j$, then add $G$ to the list \ref{item:I1} of $P_{n+1}$;
                \item otherwise add $G$ and $r_{F_n}(G)$ to the list \ref{item:I1} of $P_{n+1}$.
            \end{itemize}
            \item otherwise add $G$ and $r_{F_n}(G)$ to the list \ref{item:I1} of $P_{n+1}$.
        \end{itemize}
        \item[\ref{item:I2}] The vertices of the two graphs of type $3$ and $7$ are the facets of type $3$ and $7$ in \ref{item:I1} of $P_{n+1}$, respectively. The edges of the graphs are obtained as follows.
        \begin{itemize}
            \item If in \ref{item:I2} of $P_n$ we have (F,G;k) then:
            \begin{itemize}
            \item If in \ref{item:I1} of $P_{n+1}$ we have $F,r_{F_n}(F),G,r_{F_n}(G)$, then in \ref{item:I2} of $P_{n+1}$ we add the edges $(F,G;k)$ and $(r_{F_n}(F),r_{F_n}(G);k)$;
                \item If in \ref{item:I1} of $P_{n+1}$ we have $F,r_{F_n}(F),G$ and not $r_{F_n(G)}$, then in \ref{item:I2} of $P_{n+1}$ we add the edges $(F,G;k)$ and $(r_{F_n}(F),G;k)$;
            \item If in \ref{item:I1} we have $F, G$ and not $r_{F_n}(F),r_{F_n}(G)$, then in \ref{item:I2} of $P_{n+1}$ we add $(F,G;k)$.
            \end{itemize}
            \item Let $F$ be a type-$3$ (or type-$7$) facet of $P_{n+1}$. If in \ref{item:I1} of $P_{n+1}$ we have $F$ and $r_{F_n}(F)$, the label $F$ is in the picture of the facet $F_n$ of $L_n$ in \ref{item:I3} of $P_n$, the facet $F$ is of type $i$, the facet $F_n$ is of type $j$ and the label of the edge between $i$ and $j$ in the Coxter diagram of $P_0$ is $2k$, then in \ref{item:I2} of $P_{n+1}$ we add $(F,r_{F_n}(F);k)$.
        \end{itemize}
        \item[\ref{item:I3}] Let $G\neq F_n$ be a facet of $L_n$. 
        \begin{itemize}
            \item If the picture of $G$ does not contain the label $F_n$, then in $P_{n+1}$ add the pictures of the facets $G$ and $r_{F_n}(G)$. For the picture of $G$ of $P_{n+1}$ we copy the one of $P_n$. For the picture of $r_{F_n}(G)$ of $P_{n+1}$ we copy the picture of $G$ of $P_n$ and, for every facet $F$ such that $r_{F_n}(F)$ is in \ref{item:I1} of $P_{n+1}$, we replace the label $F$ with $r_{F_n}(F)$. (An example is shown in Figure \ref{f3}.)

            \begin{figure}
            \includegraphics{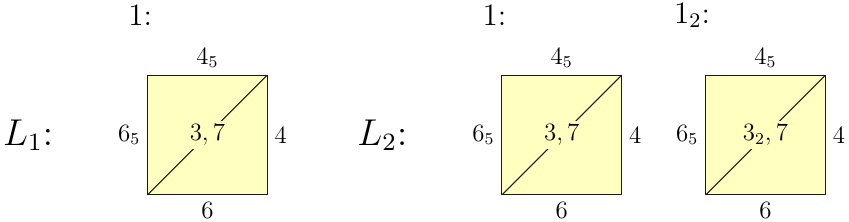}
                \caption{The facet $1$ of $L_1$ (left) and the facets $1$ and $1_2$ of $L_2$ (right).}\label{f3}
            \end{figure}

            \item If the picture of $G$ has the label $F_n$ near one edge, then in $P_{n+1}$ we add a picture of $G$ that is the double of the picture of $G$ of $P_n$ along the edge with label $F_n$, without reporting any label. Outside of the picture, near the edge that was present also in $P_n$ we put the same label, say $K$. Near the two edges that we doubled, we put the same label as before of doubling. Near the last edge, we put $r_{F_n}(K)$. The picture is tessellated in two copies $C_1$ and $C_2$ of the picture $G$ of $P_n$. In the copy $C_1$ corresponding to the one in $P_n$ we copy the labels inside the picture of $G$ in $L_n$. In the other copy $C_2$, we put the same labels, in a way that the resulting labels are symmetric with respect to the edge along which we doubled the picture, and now, for every label $K$ in $C_2$, if $r_{F_n}(K)$ is in \ref{item:I1} of $P_{n+1}$, we replace $K$ with $r_{F_n}(K)$. (An example is shown in Figure \ref{f4}.) 

            \begin{figure}
            \includegraphics{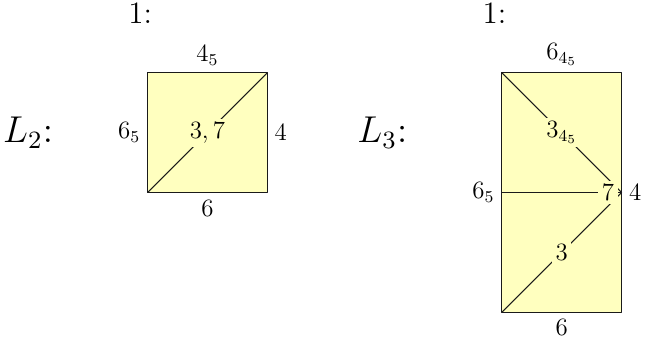}
                \caption{The facet $1$ of $L_2$ (left) and the facet $1$ of $L_3$ (right).}\label{f4}
            \end{figure}
        \end{itemize}
        
    \end{enumerate}
\end{prop}
\begin{proof}
We divide the proof in the same cases of the statement.
    \begin{enumerate}
        \item[\ref{item:I1}] 
        \begin{itemize}
            \item The label $G$ is  in the picture of $F_n$ if and only if the facet $G$ meets $F_n$ in $P_n$.
            \begin{itemize}
                \item If $G$ is of the same type of $F_n$, then the two facets are both non-compact, hence by Remark \ref{cubo} the dihedral angle between them is $\frac{\pi}{2}$. Hence $G\cup r_{F_n}(G)$ is a facet of $P_{n+1}$, that we call $G$. Hence we add $G$ to the list.
                \item If $G$ and $F_n$ are of type $i$ and $j$, respectively, with $i\ne j$, then by Proposition \ref{angolo} the dihedral angle between the two facets is the same dihedral angle between the facets \textit{\textbf{i}} and \textit{\textbf{j}} of $P_0$. There is no edge between the vertices $i$ and $j$ if and only if the dihedral angle between the facets $\textit{\textbf{i}}$ and $\textit{\textbf{j}}$ of $P_0$ is $\frac{\pi}{2}$. In this case $G\cup r_{F_n}(G)$ is a facet of $P_{n+1}$, that we call $G$. Hence we add $G$ to the list.
                \item Otherwise, if $F_n$ and $G$ are of type $i$ and $j$, respectively, with $i\ne j$, and there is an edge between the vertices $i$ and $j$, then by Proposition \ref{angolo} the dihedral angle between $G$ and $F_n$ is $\frac{\pi}{k}$, with $k\ne 2$; hence we have two facets of $P_{n+1}$ named $G$ and $r_{F_n}(G)$. Hence we add $G$ and $r_{F_n}(G)$ to the list.
            \end{itemize}
            \item Otherwise, if the label $G$ is not in the picture in \ref{item:I3} of $P_n$, then the facets $G$ and $F_n$ of $P_n$ do not meet. 
            Hence we have two facets of $P_{n+1}$ named $G$ and $r_{F_n}(G)$.
        \end{itemize}
        In this way we have listed all the facets of $P_{n+1}$. Indeed the union of all the listed facets is equal to the union of all facets of $P_n$ and of $r_{F_n}(P_n)$, minus the facet $F_n$.
        \item[\ref{item:I2}] The vertices of the two graphs are the facets of type $3$ and $7$ in the list of facets \ref{item:I1} of $P_{n+1}$ by definition.
        \begin{itemize}
            \item If in \ref{item:I2} of $P_n$ we have $(F,G;k)$, then it means that the dihedral angle between the facets $F$ and $G$ of $P_n$ is $\frac{\pi}{k}$. The proof of each of the three subcases of the thesis is obvious, once noted that $r_{F_n}(G)$ is not in \ref{item:I1} of $P_{n+1}$ if and only if $(G,\frac{\pi}{2})\in I_n(F_n)$, and this holds if and only if $G\cup r_{F_n}(G)$ is a facet of $P_{n+1}$ that we call $G$.
            \item Since the label $F$ is in the picture of $F_n$ in \ref{item:I3} of $P_n$, the facet $F$ meets $F_n$ in $P_n$. Since in the Coxeter diagram of $P_0$ the edge between  $i$ and $j$ has label $2k>2$, by Proposition \ref{angolo}, the dihedral angle between $F$ and $F_n$ is $\frac{\pi}{2k}$. Hence the dihedral angle between $F$ and $r_{F_n}(F)$ in $P_{n+1}$ is $\frac{\pi}{k}$. Hence we add the edge $(F,r_{F_n}(F);k)$ to the graph.
        \end{itemize}
        By construction of $P_{n+1}$, there is no other edge to be added to the two graphs.
        \item[\ref{item:I3}]
        \begin{itemize}
            \item If the picture of $G$ does not contain the label $F_n$, the facets $G$ and $F_n$ do not meet in $P_n$, hence in $P_{n+1}$ we have the facets $G$ and $r_{F_n}(G)$. Clearly $I_n(G)=I_{n+1}(G)$ and, given a tile $K$ of the tessellation of $G$, we also have $I_n(K)=I_{n+1}(K)$ (with a little abuse, since $K$ is not a facet). Moreover, $r_{F_n}(G)$ is a copy of $G$ in $r_{F_n}(P_n)$. Hence the picture of $r_{F_n}(G)$ is the same of the picture of $G$, but if the label $M$ is present in $G$ and there is a facet $r_{F_n}(M)$ in $P_{n+1}$, then in $r_{F_n}(G)$ we replace the label $M$ with $r_{F_n}(M)$.
            \item If $G$ has the label $F_n$ near one edge, it means that the facets $G$ and $F_n$ meet in $P_n$. Since they are both non-compact, by Remark \ref{cubo} the corresponding dihedral angle is $\frac{\pi}{2}$. Hence in $P_{n+1}$ there is a facet $G^{n+1}=G^n\cup r_{F_n}(G^n)$ (we are using the same notation of the proof of Proposition \ref{angolo}). The picture of the facet $G$ of $L_{n+1}$ is obtained doubling  the picture of $G$ of $L_n$ along the edge with label $F_n$.
        \end{itemize}
    \end{enumerate}
\end{proof}

\subsection{The construction}\label{costruzione}
We are now ready to build our sequence of polytopes.

Recall that we want $P_{n+1}=P_n\cup r_{F_n}(P_n)$, where $F_n$ is a non-compact and admissible facet of $P_n$. For every $n$ we are going to choose as $F_n$ a facet of different type from $3$ and $7$. Such a facet is non-compact and admissible in $P_0$ since the only compact facets are $\textbf{3}$ and $\textbf{7}$, and every facet of $P_0$ is admissible. For every $n\ge 1$ such a facet is non-compact and admissible by Corollary \ref{suit}.

\begin{defn}
We define the following polytopes: $P_1=P_0\cup r_5(P_0)$, $P_2=P_1\cup r_2(P_1)$, $P_3=P_2\cup r_{4_5}(P_2)$, $P_4=P_3\cup r_4(P_3)$, $P_5=P_4\cup r_1(P_4)$, $P_6=P_5\cup r_6(P_5)$, $P_7=P_6\cup r_{6_5}(P_6)$, $P_8=P_7\cup r_{1_2}(P_7)$.
\end{defn}

\begin{prop}
    The information \ref{item:I1}, \ref{item:I2}, \ref{item:I3} on $P_n$, for $n=0,\dots, 8$, are the following.

    The information \ref{item:I1} is:
    \begin{itemize}
        \item[$P_0:$] $1,2,3,4,5,6,7$;
        \item[$P_1:$]       $1,2,3,4, 4_5,6, 6_5,7$;
        \item[$P_2:$]       $1,1_2,3, 3_2, 4, 4_5, 6, 6_5, 7$;
        \item[$P_3:$]   $1,1_2,3, 3_{4_5}, 3_2,3_{4_5,2},   4, 6, 6_{4_5}, 6_5, 7$
        \item[$P_4:$]       $1,1_2,3, 3_4,3_{4_5}, 3_{4,4_5},   3_2,3_{4,2},3_{4_5,2},              3_{4,4_5,2},
            4,6,6_{4_5}, 6_5, 6_{4,5},7$;
        \item[$P_5:$]       $1_2,1_{1,2},3, 3_4,3_{4_5},        3_{4,4_5}, 3_2, 3_{1,2},            3_{4,2},3_{1,4,2}, 3_{4_5,2}, 3_{1,4_5,2}, 3_{4,4_5,2}, 3_{1,4,4_5,2},
        6,6_{4_5},
        \\ 6_5, 6_{4,5},7,7_1$;
        \item[$P_6:$] $1_2,1_{1,2},3, 3_4,3_{4_5}, 3_{6,4_5}, 3_{4,4_5}, 3_{6,4,4_5}, 3_2, 3_{1,2},
    3_{4,2},3_{1,4,2}, 3_{4_5,2}, 3_{6,4_5,2}, 3_{1,4_5,2}, 
    \\ 3_{6,1,4_5,2}, 3_{4,4_5,2}, 3_{6,4,4_5,2}, 3_{1,4,4_5,2}, 3_{6,1,4,4_5,2}, 6_{4_5}, 6_{6,4_5}, 6_5, 6_{4,5},7, 7_6, 7_1, 7_{6,1}$;
        \item[$P_7:$] $1_2,1_{1,2},3, 3_4, 3_{6_5,4},3_{4_5}, 3_{6,4_5}, 3_{4,4_5}, 3_{6_5,4,4_5},  3_{6,4,4_5}, 3_{6_5,6,4,4_5}, 3_2, 3_{1,2},
    3_{4,2}, 
    \\ 3_{6_5,4,2}, 
    3_{1,4,2}, 3_{6_5,1,4,2}, 3_{4_5,2}, 3_{6,4_5,2}, 3_{1,4_5,2}, 3_{6,1,4_5,2}, 3_{4,4_5,2}, 3_{6_5,4,4_5,2}, 3_{6,4,4_5,2},   3_{6_5,6,4,4_5,2},  
    3_{1,4,4_5,2}, 
    \\ 3_{6_5,1,4,4_5,2}, 3_{6,1,4,4_5,2}, 3_{6_5,6,1,4,4_5,2}, 6_{4_5}, 6_{6,4_5}, 6_{4,5}, 6_{6_5,4,5}, 7, 7_{6_5}, 7_6, 7_{6_5,6}, 7_1, 7_{6_5,1}, 7_{6,1}, 7_{6_5,6,1}$;
    \item[$P_8:$] $1_{1,2}, 1_{1_2,1,2}, 3, 3_{1_2}, 3_4, 3_{1_2,4}, 3_{6_5,4}, 3_{1_2,6_5,4}, 3_{4_5}, 3_{1_2,4_5}, 3_{6,4_5}, 3_{1_2,6,4_5}, 3_{4,4_5}, 
    \\ 3_{1_2,4,4_5}, 3_{6_5,4,4_5}, 3_{1_2,6_5,4,4_5},  3_{6,4,4_5}, 3_{1_2,6,4,4_5}, 
    3_{6_5,6,4,4_5}, 3_{1_2,6_5,6,4,4_5}, 3_2, 3_{1,2}, 3_{1_2,1,2}, 3_{4,2}, 3_{6_5,4,2}, 
    \\ 3_{1,4,2}, 3_{1_2,1,4,2}, 3_{6_5,1,4,2}, 3_{1_2,6_5,1,4,2}, 3_{4_5,2}, 3_{6,4_5,2},  3_{1,4_5,2}, 3_{1_2,1,4_5,2}, 3_{6,1,4_5,2}, 3_{1_2,6,1,4_5,2}, 3_{4,4_5,2}, 
    \\ 3_{6_5,4,4_5,2}, 3_{6,4,4_5,2},  3_{6_5,6,4,4_5,2}, 
    3_{1,4,4_5,2}, 3_{1_2,1,4,4_5,2}, 3_{6_5,1,4,4_5,2}, 3_{1_2,6_5,1,4,4_5,2}, 3_{6,1,4,4_5,2}, 
    \\ 3_{1_2,6,1,4,4_5,2}, 3_{6_5,6,1,4,4_5,2}, 3_{1_2,6_5,6,1,4,4_5,2}, 6_{4_5}, 6_{6,4_5}, 6_{4,5}, 6_{6_5,4,5},
    7, 7_{1_2}, 7_{6_5}, 7_{1_2,6_5}, 7_6, 7_{1_2,6}, 7_{6_5,6}, 
    \\ 7_{1_2,6_5,6}, 7_1, 7_{1_2,1}, 7_{6_5,1}, 7_{1_2,6_5,1}, 7_{6,1}, 7_{1_2,6,1}, 7_{6_5,6,1}, 7_{1_2,6_5,6,1}$.
    \end{itemize}
    
    The information \ref{item:I2} is in Tables \ref{t0}$,\dots,$ \ref{t8}.

    The information \ref{item:I3} is in Figure \ref{f2} and in Figures \ref{f7}$,\dots,$ \ref{f13.4}.
\end{prop}
\begin{proof}
    The information on $P_0$ can be easily recovered from the definition of $P_0$.
    We have shown in Section \ref{P7eP8} the information on $P_1$. We now recover the information on $P_{n+1}$, for $n\ge 1$, from the information on $P_n$ and the Coxeter diagram on $P_0$, using Proposition \ref{procedura}. 
    
    We now describe in detail how to recover the information on $P_2$. The reader is invited to check the information on the other polytopes in the same way.

    The information on $P_1$ are the following.

\begin{enumerate}
    \item[\ref{item:I1}] Facets of $P_1$: $1,2,3,4, 4_5,6, 6_5,7$.
    \item[\ref{item:I2}] Adjacency matrices of facets of type $3$ and $7$ are in Table \ref{t1}. They are clearly two $1 \times 1$ matrices, both with the entry $1$. (Recall that on the left side of the matrices we put the names of the facets.)

\item[\ref{item:I3}] The pictures of the facets of $\lk_1$ are in Figure \ref{f2}.
\end{enumerate}

We now use Proposition \ref{procedura} to recover the information on $P_2$.

\begin{enumerate}
    \item[\ref{item:I1}] The label $1$ is not in the picture of the facet $2$ in \ref{item:I3} of $P_1$, hence we add $1$ and $1_2$ to the list of facets \ref{item:I1} of $P_2$.

    The label $3$ is in the picture of the facet $2$ in \ref{item:I3} of $P_1$. Moreover the facets $\textbf{2}$ and $\textbf{3}$ are of different type and there is an edge between the corresponding vertices in the Coxeter diagram of $P_0$ (see the diagram in Section \ref{P0}). Hence we add $3$ and $3_2$ to the list of facets of $P_2$.

    The label $4$ is in the picture of the facet $2$ in \ref{item:I3} of $P_1$. Moreover the facets $\textbf{2}$ and $\textbf{3}$ are of different type and there is not an edge between the corresponding vertices in the Coxeter diagram of $P_0$. Hence we add $4$ to the list of facets of $P_2$. The same holds for the remaining facets ($\textbf{4}_5,\textbf{6},\textbf{6}_5,\textbf{7}$) distinct to $\textbf{2}$ of $P_1$.

    We obtained that the list of facets of $P_2$ is $1,1_2,3,3_2,4,4_5,6,6_5,7$.

    \item[\ref{item:I2}] The vertices of the two adjacency graphs of $P_2$ are the facets of type $3$ or $7$ in \ref{item:I1} of $P_2$: $3, 3_2$ and $7$. The two graphs in \ref{item:I2} of $P_1$ have no edge. We have $3$ and $3_2$ in \ref{item:I1} of $P_{2}$, the label $3$ is in the picture of the facet $\textbf{2}$ in \ref{item:I3} of $P_1$, the facet $\textbf{3}$ is of type $3$, the facet $\textbf{2}$ is of type $2$ and the label of the edge between $3$ and $2$ in the Coxeter diagram of $P_0$ is $4$. Hence we add an edge with label $2$ between the vertices $3$ and $3_2$.

    We obtained that the two adjacency matrices of $P_2$ are the ones in Table \ref{t2}.

\item[\ref{item:I3}] The picture in \ref{item:I3} of $P_1$ of the facet $\textbf{1}$ of $L_1$ does not contain the label $2$. Hence we add the the first two pictures of Figure \ref{f7}.
The pictures in \ref{item:I3} of $P_1$ of the facets $\textbf{4},\textbf{4}_5,\textbf{6},\textbf{6}_5$ of $L_1$ contain the label $2$. Hence we add the latter four pictures of Figure \ref{f7}.
\end{enumerate}
\end{proof}

\section{The reflectofolds}\label{reflect}

In this section we glue the facets of the polytopes $P_7$ and $P_8$ in order to obtain some $1$-cusped developable reflectofolds. In Section \ref{ref1} we perform the gluing. Then, in Section \ref{ref2} we study the facets and the corners of the constructed spaces, in order to show, in Section \ref{ref3}, that they are $1$-cusped developable reflectofolds.

\subsection{Defining the reflectofolds}\label{ref1}
The link $L_7$ of the ideal vertex of $P_7$ is a right parallelepiped. If we glue $L_7$ as described in Figure \ref{f14}, in each of the three cases we obtain a flat $3$-manifold: the $3$-torus, the $\frac{1}{2}$-twist manifold  and the $\frac{1}{4}$-twist manifold, respectively \cite[Figure 12.2]{MA}.

\begin{figure}
\includegraphics{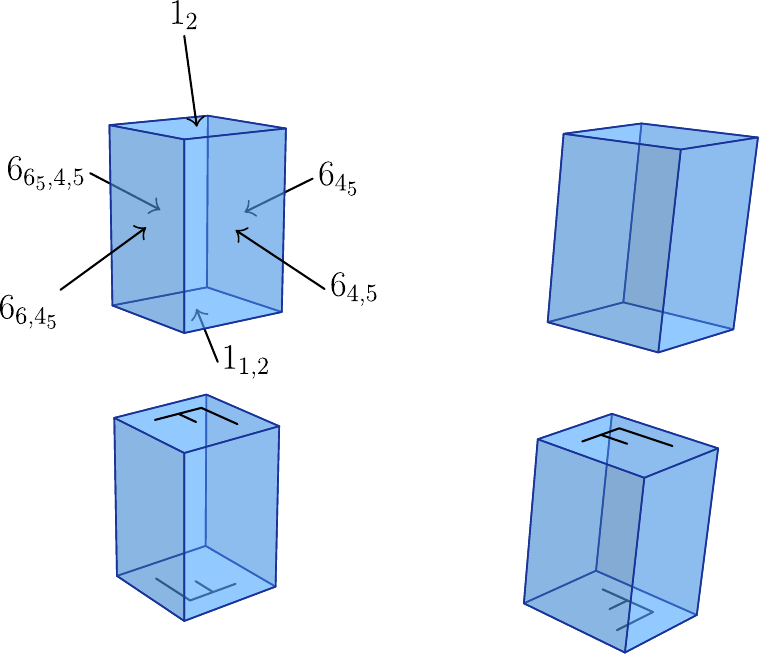}
    \caption{The link $L_7$ (top-left), the $3$-torus (top-right), the $\frac{1}{2}$-twist manifold (bottom-left), the $\frac{1}{4}$-twist manifold (bottom-right). In the last three pictures, if two opposite facets do not have  a letter inside, we glue them with a translation, otherwise we glue them as indicated with the letters.}\label{f14}
\end{figure}

We now show that, for each of the three manifolds, we can glue $P_7$ using isometries between the facets in a way that this induces a gluing of $L_7$ as described.

Let $R_T$ be the space obtained from $P_7$ by gluing the facet $6_{4_5}$ with $6_{6,4_5}$ using the isometry $\restr{r_6}{6_{4_5}}$, the facet $6_{4,5}$ with $6_{6_5,4,5}$ using the isometry $\restr{r_{6_5}}{6_{4,5}}$, and the facet $1_2$ with $1_{1,2}$ using the isometry $\restr{r_1}{1_2}$. We have indeed $r_6(6_{4_5})  = 6_{6,4_5}$, $r_{6_5}(6_{4,5}) = 6_{6_5,4,5}$ and $r_1(1_2) = 1_{1,2}$. This can be seen from Figure \ref{f15} for $L_7$, and therefore it also holds for $P_7$ since each map is a reflection through a copy of a facet of $P_0$.

In the next cases the argument is analog to the one of $R_T$.

\begin{defn}
The space $R_T$ is obtained from $P_7$ by gluing the facets via the following isometries:

$$
\restr{r_6}{6_{4_5}}\colon 6_{4_5}\rightarrow 6_{6,4_5}, \quad \restr{r_{6_5}}{6_{4,5}}\colon 6_{4,5}\rightarrow 6_{6_5,4,5}, \quad \restr{r_1}{1_2}\colon 1_2\rightarrow 1_{1,2}.
$$  

Let $R_{\frac{1}{2}}$ be the space obtained from $P_7$ by gluing the facets via the following isometries:
$$
\restr{r_6}{6_{4_5}}\colon 6_{4_5}\rightarrow 6_{6,4_5}\quad \restr{r_{6_5}}{6_{4,5}}\colon 6_{4,5}\rightarrow 6_{6_5,4,5}, \quad \restr{r_1\circ r_6 \circ r_{6_5}}{1_2}\colon 1_2\rightarrow 1_{1,2}.
$$

Let $R_{\frac{1}{4}}$ be the space obtained from $P_7$ by gluing the facets via the following isometries:
$$
\restr{r_6}{6_{4_5}}\colon 6_{4_5}\rightarrow 6_{6,4_5}\quad \restr{r_{6_5}}{6_{4,5}}\colon 6_{4,5}\rightarrow 6_{6_5,4,5}, \quad \restr{r_1\circ r_6 \circ r_5}{1_2}\colon 1_2\rightarrow 1_{1,2}.
$$
\end{defn}

\begin{figure}
\includegraphics{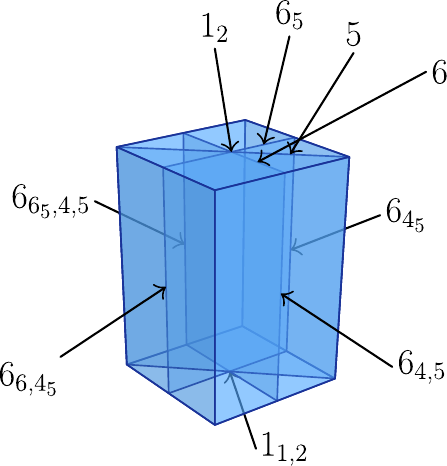}
    \caption{The link $L_7$ and the fixed planes of the reflections used to define $R_T,R_{\frac{1}{2}},R_{\frac{1}{4}}$.}\label{f15}
\end{figure}

We see from Figure \ref{f15} that each gluing induces a gluing of $L_7$ as in Figure \ref{f14}, thus producing the $3$-torus, the $\frac{1}{2}$-twist manifold and the $\frac{1}{4}$-twist manifold, respectively.

The link $L_8$ of the ideal vertex of $P_8$ is a right parallelepiped. If we glue $L_8$ as described in Figure \ref{f16}, we obtain a flat $3$-manifold, the Hantzsche-Wendt manifold \cite[Figure 12.2]{MA}.
\begin{figure}
\includegraphics{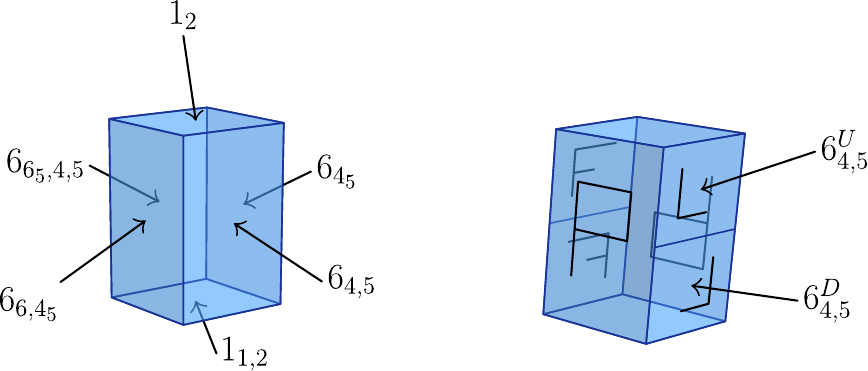}
    \caption{The link $L_8$ (left) and the Hantzsche-Wendt manifold (right), with the same notation of Figure \ref{f14}. Moreover, we see how the facet $6_{4,5}$ is divided in the two parts $6_{4,5}^U$ and $6_{4,5}^D$. Similarly the facet $6_{6_5,4,5}$ is divided in the two parts $6_{6_5,4,5}^U$ and $6_{6_5,4,5}^D$}\label{f16}
\end{figure}
We now show that we can glue $P_8$ using isometries between the facets in a way that this induces the gluing of $L_8$ described in Figure \ref{f16}.

We notice that the facet $6_{4,5}$ is divided in two parts, $6_{4,5}^U$ and $6_{4,5}^D$, as in Figure \ref{f16}. Similarly, we define $6_{6_5,4,5}^U$ and $6_{6_5,4,5}^D$.

\begin{defn}
    Let $R_{HW}$ be the space obtained from $P_8$ by gluing the facets via the following isometries:
    $$
        \restr{r_{1_2}}{1_{1,2}}\colon 1_{1,2}\rightarrow 1_{1_2,1,2}, \quad \restr{r_{1_2} \circ r_{6_5} \circ r_{6}}{6_{4_5}}\colon 6_{4_5}\rightarrow 6_{6,4_5},
    $$
    $$
        \restr{r_{6} \circ (r_1 \circ r_2)^2}{6_{4,5}^U}\colon 6_{4,5}^U\rightarrow 6_{4,5}^D, \quad
        \restr{r_{6} \circ (r_1 \circ r_2)^2}{6_{6_5,4,5}^U}\colon 6_{6_5,4,5}^U\rightarrow 6_{6_5,4,5}^D
    $$
\end{defn}


\begin{figure}
\includegraphics{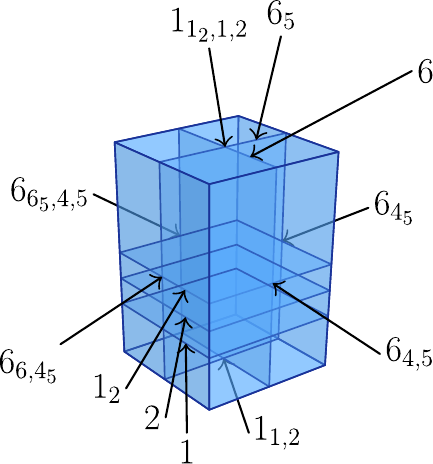}
    \caption{The link $L_8$ with the fixed planes of the reflections used to define $R_{HW}$.}\label{f17}
\end{figure}

We see from Figure \ref{f17} that this gluing induces the gluing of $L_8$ described in Figure \ref{f16}.

\begin{rem}
\label{identity}
    Let $f$ be one of the gluing maps used above for the polytope $P_7$. Then $f$ is the restriction of a symmetry of $P_7$ that preserves its tessellation in copies of $P_0$. Indeed, we see from Figure \ref{f15} that $f(L_7)=L_7$, hence it easily follows that $f(P_7)=P_7$. Moreover, $f$ is a composition of reflections along copies of facets of $P_0$. Hence it is a symmetry of $P_7$ and preserves the tessellation.

    If $f$ is a gluing map for the polytope $P_8$, the statement is slightly different. Indeed, if we consider the natural tessellation of $\matR^3$ in copies of $L_8$, then $f$ is induced by a symmetry of $\matR^3$ that preserves its tesselletion in copies of $L_0$. The argument is analog to the previous case.
\end{rem}

Let $R$ be any of $R_T$, $R_{\frac{1}{2}}$, $R_{\frac{1}{4}}, R_{HW}$. The purpose of the following sections will be to prove this theorem.

\begin{thm}
\label{Rbello}
    The space $R$ is an orientable, finite-volume, $1$-cusped, developable reflectofold with compact, non-empty boundary. Moreover, the cusp of $R_T, R_{\frac{1}{2}},R_{\frac{1}{4}}, R_{HW}$ has section the $3$-torus, the $\frac{1}{2}$-twist manifold, the $\frac{1}{4}$-twist manifold, the Hantzsche-Wendt manifold, respectively.
\end{thm}

The proof of Theorem \ref{thm:main} will immediately follow from Theorem \ref{Rbello} and Corollary \ref{maincor}.

\subsection{The facets and the corners.}\label{ref2}
The purpose of this section is to study the facets and the corners of $R$. This will help us to prove Theorem \ref{Rbello} in the next section.


Let $P$ be any of $P_7$ and $P_8$, and $p\colon P\rightarrow R$ denote the quotient map.

\begin{lemma}
\label{tipi}
    A facet of $R$ is:
    \begin{itemize}
        \item either the image through $p$ of a facet of type $7$,
        \item or the image through $p$ of a union of facets of type $3$.
    \end{itemize}
\end{lemma}
We call the first facets of $R$ \emph{of type} $7$ and the other facets \emph{of type} $3$.
\begin{proof}
    Since we glued all the facets of different type from $3$ and $7$, the union of the facets of $R$ is the image through $p$ of the union of the facets of $P$ of type $3$ and $7$.

    If in $P$ a facet of type $3$ and a facet along which we glue meet, they do so with a dihedral angle of $\frac{\pi}{2}$. This is true by Proposition \ref{angolo}, since we glue facets that are of type $1$ and $6$ and in $P_0$ the facets $\textbf{1}$ and $\textbf{6}$ are orthogonal to $\textbf{3}$.
    Let $A$ and $B$ be two facets of $P$ that are identified in $R$ via the gluing.
    By Remark \ref{identity}, if $F$ and $G$ are facets of $P$ of type $3$ or $7$ such that $p(F\cap A)=p(G\cap B)\ne \emptyset$, then $F$ and $G$ are of the same type.

    Let $S_A$ and $S_B$ be the sets of facets of $P$ of type $3$ that meet $A$ and $B$, respectively. Then given $F\in S_A$, there exists $G\in S_B$ such that $p(F\cap A)=p(G\cap B)$. Then $p(F)$ and $p(G)$ are contained in the same facet of $R$ (since we have already seen that the dihedral angle in $P$ between $F$ and $A$, and $G$ and $B$, is $\frac{\pi}{2}$).
    Since the image through $p$ of a facet is contained in a facet of $R$, we have shown that a facet of $R$ is the image through $p$ of a union  of facets of the same type: either $7$ or $3$. It thus only remains to show that in the type-$7$ case such a facet is the image of exactly one facet of $P$.


    If in $P$ a facet of type $7$ and a facet along which we glue meet, they do so with a dihedral angle different from $\frac{\pi}{2}$. Indeed this is true by Proposition \ref{angolo}, since we glue facets that are of type $1$ and $6$ and in $P_0$ the facets $\textbf{1}$ and $\textbf{6}$ are not orthogonal to $\textbf{7}$.

    Hence for every type-$7$ facet $M$ of $P$, we obtain that $p(M)$ is a facet of $R$.
\end{proof}

It will be easy to check that $R$ satisfies \ref{AC} and \ref{EF}  once we have found the corner graphs of type $3$ or $7$ of $R$.
\begin{defn}
\label{defcgraph}
    For $i=3,7$, the \emph{type-$i$ corner graph} $G_i$ of $R$ is the graph whose vertices are the type-$i$ facets of $R$ and between two vertices $A$ and $B$ there is an edge for every corner in $A\cap B$. Moreover, we put a label $k\in\matN$ on an edge if the dihedral angle associated to the corresponding corner is $\frac{\pi}{k}$. If the angle is not in the form $\frac{\pi}{k}$ (this will never be the case), then we put the underlined angle as a label. 
\end{defn}

By Lemma \ref{tipi}  we already know the vertices of the type-$7$ corner graph of $R$.

We call a ridge of $P$ \emph{of type $(i,j)$} if it is the intersection of a facet of type $i$ and a facet of type $j$.
We call a corner of $R$ \emph{of type $(i,j)$} if it is contained in the intersection of two facets, one of type $i$ and one of type $j$.

\begin{lemma}
\label{corner}
    A corner of $R$ is:
    \begin{itemize}
        \item either of type $(3,3)$, and in this case it is the image through $p$ of a union of some type-$(3,3)$ ridges;
        \item either of type $(7,7)$, and in this case it is:
        \begin{itemize}
            \item either the image through $p$ of a type-$(7,7)$ ridge of $P$;
            \item or the image through $p$ of a type-$(7,i)$ ridge of $P$, with $i=1,6$;
        \end{itemize}
        
        \item or of type $(3,7)$, and in this case it is the image through $p$ of a type-$(3,7)$ ridge of $P$.
    \end{itemize}
\end{lemma}
\begin{proof}
    Since the facets of $R$ are of type $3$ or $7$, there are three kinds of corners in $R$: type $(3,3)$, $(7,7)$, and $(3,7)$.

    By Proposition \ref{angolo}, if a facet of type $3$ and a facet of type $1$ or $6$ of $P$ meet, the dihedral angle between them is $\frac{\pi}{2}$. Hence the image through $p$ of a type-$(3,i)$ ridge of $P$, with $i=1,6$, is contained in the relative interior of a type-$3$ facet. Hence the union of the type-$(3,3)$ corners of $R$ is the image through $p$ of the union of the type-$(3,3)$ ridges of $P$.

    The image through $p$ of a type-$(3,3)$ ridge is contained in a corner, hence every type-$(3,3)$ corner is the image through $p$ of the union of some type-$(3,3)$ ridges.
    

    Since by Lemma \ref{tipi} the image through $p$ of a facet of type $7$ of $P$ is a facet of type $7$ of $R$, the image of a type-$(7,7)$, or type-$(3,7)$, ridge is a corner of $R$. Moreover,  the image of a type-$(7,i)$ ridge, with $i=1,6$, is a type-$(7,7)$ corner.
\end{proof}

Let $3_X,3_Y$ be two type-$3$ facets of $P$. Let us define the following equivalence relation: we set $3_X \sim 3_Y$ if $p(3_X)$ and $p(3_Y)$ are contained in the same facet of $R$.
Moreover, the type-$3$ facets of $R$ are in natural bijection with the equivalence classes. Indeed,  $\overline{3_X}=\{3_{X_1},\dots,3_{X_k}\}$ is an equivalence class if and only if $\bigcup_{i=1}^k p(3_{X_i})$ is a facet of $R$.

Since by Lemma \ref{tipi} the map $p$ gives a correspondence between the type-$7$ facets of $P$ and the type-$7$ facets of $R$, we will call the type-$7$ facets of $R$ with the same name of the ones of $P$. Instead we will call the type-$3$ facets of $R$ with the same name of the equivalence classes.

\begin{lemma}
\label{facce}
    The equivalence classes are:
    \begin{itemize}
        \item $R_T:$
        \begin{itemize}
            \item $\overline{3_{4,4_5,2}}=\{3_{4,4_5,2},3_{6,4,4_5,2},3_{6_5,6,4,4_5,2},3_{6_5,4,4_5,2},3_{1,4,4_5,2},3_{6,1,4,4_5,2},3_{6_5,1,4,4_5,2}.3_{6_5,6,1,4,4_5,2}\}$;
            \item $\overline{3_{4,2}}=\{3_{4,2},3_{6_5,4,2},3_{1,4,2},3_{6_5,1,4,2}\}$;
            \item $\overline{3_4}=\{3_4,3_{6_5,4}\}$;
            \item $\overline{3_{4_5}}=\{3_{4_5},3_{6,4_5}\}$;
            \item $\overline{3_2}=\{3_2,3_{1,2}\}$;
            \item $\overline{3_{4,4_5}}=\{3_{4,4_5},3_{6,4,4_5},3_{6_5,4,4_5},3_{6_5,6,4,4_5}\}$;
            \item $\overline{3_{4_5,2}}=\{3_{4_5,2},3_{6,4_5,2},3_{1,4_5,2},3_{6,1,4_5,2}\}$
            \item $\overline{3}=\{3\}$;
        \end{itemize}
        \item $R_{\frac{1}{2}}:$
        \begin{itemize}
            \item $\overline{3_{4,4_5,2}}=\{3_{4,4_5,2},3_{6,4,4_5,2},3_{6_5,6,4,4_5,2},3_{6_5,4,4_5,2},3_{1,4,4_5,2},3_{6,1,4,4_5,2},3_{6_5,1,4,4_5,2}.3_{6_5,6,1,4,4_5,2}\}$;
            \item $\overline{3_{4,2}}=\{3_{4,2},3_{6_5,4,2},3_{1,4,2},3_{6_5,1,4,2}\}$;
            \item $\overline{3_4}=\{3_4,3_{6_5,4}\}$;
            \item $\overline{3_{4_5}}=\{3_{4_5},3_{6,4_5}\}$;
            \item $\overline{3_2}=\{3_2,3_{1,2}\}$;
            \item $\overline{3_{4,4_5}}=\{3_{4,4_5},3_{6,4,4_5},3_{6_5,4,4_5},3_{6_5,6,4,4_5}\}$;
            \item $\overline{3_{4_5,2}}=\{3_{4_5,2},3_{6,4_5,2},3_{1,4_5,2},3_{6,1,4_5,2}\}$
            \item $\overline{3}=\{3\}$;
        \end{itemize}
        \item $R_{\frac{1}{4}}:$
        \begin{itemize}
            \item $\overline{3_{4,4_5,2}}=\{3_{4,4_5,2},3_{6,4,4_5,2},3_{6_5,6,4,4_5,2},3_{6_5,4,4_5,2},3_{1,4,4_5,2},3_{6,1,4,4_5,2},3_{6_5,1,4,4_5,2}.3_{6_5,6,1,4,4_5,2}\}$;
            \item $\overline{3_{4,2}}=\{3_{4,2},3_{6_5,4,2},3_{1,4_5,2},3_{6,1,4_5,2}\}$;
            \item $\overline{3_4}=\{3_4,3_{6_5,4}\}$;
            \item $\overline{3_{4_5}}=\{3_{4_5},3_{6,4_5}\}$;
            \item $\overline{3_2}=\{3_2,3_{1,2}\}$;
            \item $\overline{3_{4,4_5}}=\{3_{4,4_5},3_{6,4,4_5},3_{6_5,4,4_5},3_{6_5,6,4,4_5}\}$;
            \item $\overline{3_{4_5,2}}=\{3_{4_5,2},3_{6,4_5,2},3_{1,4,2},3_{6_5,1,4,2}\}$
            \item $\overline{3}=\{3\}$;
        \end{itemize}
        \item $R_{HW}:$
        \begin{itemize}
            \item[-] $\overline{3_{4,4_5,2}}=\{3_{4,4_5,2},3_{6_5,1,4,4_5,2},3_{1_2,6_5,1,4,4_5,2},3_{6_5,6,4,4_5,2},3_{6,1,4,4_5,2},3_{1_2,6,1,4,4_5,2}\}$;
            \item[-] $\overline{3_{1,4_5,2}}=\{3_{1,4_5,2},3_{1_2,1,4_5,2},3_{6,1,4_5,2},3_{1_2,6,1,4_5,2}\}$;
            \item[-] $\overline{3_{1,4,4_5,2}}=\{3_{1,4,4_5,2},3_{1_2,1,4,4_5,2},3_{6,4,4_5,2},3_{6_5,6,1,4,4_5,2},3_{6_5,4,4_5,2},3_{1_2,6_5,6,1,4,4_5,2}\}$;
            \item[-] $\overline{3_{1,2}}=\{3_{1,2},3_{1_2,1,2}\}$;
            \item[-] $\overline{3_{4,2}}=\{3_{4,2},3_{1,4,2},3_{1_2,1,4,2}\}$;
            \item[-] $\overline{3_{4,4_5}}=\{3_{4,4_5},3_{1_2,6,4,4_5},3_{6_5,4,4_5},3_{1_2,6_5,6,4,4_5}\}$;
            \item[-] $\overline{3_4}=\{3_4,3_{1_2,4}\}$;
            \item[-] $\overline{3_{6,4,4_5}}=\{3_{6,4,4_5},3_{1_2,4,4_5},3_{1_2,6_5,4,4_5},3_{6_5,6,4,4_5}\}$;
            \item[-] $\overline{3_{6_5,4}}=\{3_{6_5,4},3_{1_2,6_5,4}\}$;
            \item[-] $\overline{3_{6,4_5}}=\{3_{6,4_5},3_{1_2,4_5}\}$;
            \item[-] $\overline{3_{4_5,2}}=\{3_{4_5,2},3_{6,4_5,2}\}$;
            \item[-] $\overline{3_{4_5}}=\{3_{4_5},3_{1_2,6,4_5}\}$;
            \item[-] $\overline{3_{6_5,4,2}}=\{3_{6_5,4,2},3_{1_2,6_5,1,4,2},3_{6_5,1,4,2}\}$;
            \item[-] $\overline{3}=\{3\}$;
            \item[-] $\overline{3_{1_2}}=\{3_{1_2}\}$;
            \item[-] $\overline{3_2}=\{3_2\}$.
        \end{itemize}
    \end{itemize}
    The type-$(7,7)$ corners that are the images through $p$ of the type-$(7,i)$ ridges, with $i=1,6$, of the polytope $P$ are the following. We write $7_X \cap_k 7_Y$ to indicate a corner between the facets $7_X$ and $7_Y$ with angle $\frac{\pi}{k}$.
    \begin{itemize}
        \item[$R_T:$] $7\cap_3 7_{6_5}, 7_{1} \cap_3 7_{6_5,1}, 7_{6} \cap_3 7_{6_5,6},7_{6,1} \cap_3 7_{6_5,6,1},7 \cap_3 7_{6},7_{1} \cap_3 7_{6,1},7_{6_5} \cap_3 7_{6_5,6}, 7_{6_5,1} \cap_3 7_{6_5,6,1}, 7 \cap_2 7_{1}, 7_{6} \cap_2 7_{6,1},7_{6_5} \cap_2 7_{6_5,1},7_{6_5,6} \cap_2 7_{6_5,6,1}$;
        \item[$R_{\frac{1}{2}}:$] $7\cap_3 7_{6_5}, 7_{1} \cap_3 7_{6_5,1}, 7_{6} \cap_3 7_{6_5,6},7_{6,1} \cap_3 7_{6_5,6,1},7 \cap_3 7_{6},7_{1} \cap_3 7_{6,1},7_{6_5} \cap_3 7_{6_5,6}, 7_{6_5,1} \cap_3 7_{6_5,6,1}, 7 \cap_2 7_{6_5,6,1}, 7_{1} \cap_2 7_{6_5,6},7_{6_5} \cap_2 7_{6,1},7_{6_5,1} \cap_2 7_{6}$;
        \item[$R_{\frac{1}{4}}:$] $7\cap_3 7_{6_5}, 7_{1} \cap_3 7_{6_5,1}, 7_{6} \cap_3 7_{6_5,6},7_{6,1} \cap_3 7_{6_5,6,1},7 \cap_3 7_{6},7_{1} \cap_3 7_{6,1},7_{6_5} \cap_3 7_{6_5,6}, 7_{6_5,1} \cap_3 7_{6_5,6,1}, 7 \cap_2 7_{6,1}, 7_{6} \cap_2 7_{6_5,6,1},7_{6_5} \cap_2 7_{1},7_{6_5,6} \cap_2 7_{6_5,1}$;
        \item[$R_{HW}:$] $7\cap_3 7_{1_2,6,1}, 7_{1_2,1} \cap_3 7_{6}, 7_{1} \cap_3 7_{1_2,6},7_{1_2} \cap_3 7_{6,1},7_{6_5} \cap_3 7_{1_2,6_5,6,1},7_{1_2,6_5,1} \cap_3 7_{6_5,6},7_{1_2,6_5,6} \cap_3 7_{6_5,1}, $ $7_{1_2,6_5} \cap_3 7_{6_5,6,1}, 7_{1_2,6_5,1} \cap_3 7_{6,1}, 7_{1_2,1} \cap_3 7_{6_5,6,1}, 7_{1_2,6_5} \cap_3 7_{6}, 7_{1_2} \cap_3 7_{6_5,6}, 7_{6_5} \cap_3 7_{1_2,6}, 7 \cap_3 7_{1_2,6_5,6}, $ $ 7_{6_5,1} \cap_3 7_{1_2,6,1}, 7_{1} \cap_3 7_{1_2,6_5,6,1}, 7_{6_5,1} \cap_2 7_{1_2,6_5,1}, 7_{1} \cap_2 7_{1_2,1},7_{6_5,6,1} \cap_2 7_{1_2,6_5,6,1},7_{6,1} \cap_2 7_{1_2,6,1}$.
    \end{itemize}
\end{lemma}
\begin{proof}
    We begin with $R_T,R_{\frac{1}{2}},R_{\frac{1}{4}}$. The gluings of $6_{4,5}$ with $ 6_{6_5,4,5}$, and of $6_{4_5}$ with $6_{6,4_5}$ are in common with every case.

    \begin{itemize}
        \item \underline{$6_{4,5}$ and $6_{6_5,4,5}$}: We refer to Figure \ref{f12.1} and \ref{f18} for the information \ref{item:I3} on these two facets and the way to glue them. The latter figure is not necessary, since we can deduce its content from Figure \ref{f14}, but it helps the reader to check the results.

        \begin{figure}\label{f18}
        \includegraphics{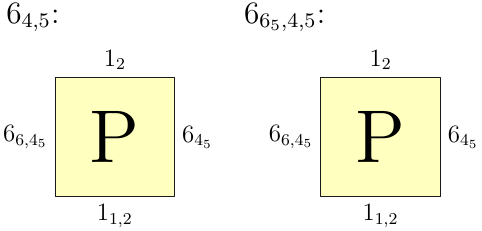}
            \caption{The way to glue the facets $6_{4,5}$ and $6_{6_5,4,5}$ of $L_7$.}
        \end{figure}



        Hence we see that:
        \begin{align*}
            \quad \quad \quad 3_{6,4,4_5,2} \sim 3_{6_5,6,4,4_5,2}; \quad \quad
            & 3_{4,2} \sim 3_{6_5,4,2}; \quad \quad
            & 3_{4,4_5,2} \sim 3_{6_5,4,4_5,2}; \\
            3_{6,4,4_5} \sim 3_{6_5,6,4,4_5}; \quad \quad
            & 3_{4} \sim 3_{6_5,4}; \quad \quad
            & 3_{4,4_5} \sim 3_{6_5,4,4_5}; \\
            3_{6,1,4,4_5,2} \sim 3_{6_5,6,1,4,4_5,2}; \quad \quad
            & 3_{1,4,2} \sim 3_{6_5,1,4,2}; \quad \quad 
            & 3_{1,4,4_5,2} \sim 3_{6_5,1,4,4_5,2}.
        \end{align*}
        

        Moreover, both $6_{4,5}$ and $6_{6_5,4,5}$ meet $4$ facets of type $7$, with a dihedral angle of $\frac{\pi}{6}$ by Proposition \ref{angolo} (since in $P_0$ the dihedral angle between $6$ and $7$ is $\frac{\pi}{6}$). 
        Hence, in $R$, from the picture we notice that there are the following corners with angle $\frac{2\pi}{6}=\frac{\pi}{3}$.
        $$
            7 \cap_3 7_{6_5}; \quad \quad
             7_1 \cap_3 7_{6_5,1}; \quad \quad
             7_6 \cap_3 7_{6_5,6}; \quad \quad
             7_{6,1} \cap_3 7_{6_5,6,1}.
        $$


    \item \underline{$6_{4_5}$ and $6_{6,4_5}$}: We refer to Figure \ref{f12.2} and \ref{f19}.
    \begin{figure}
    \includegraphics{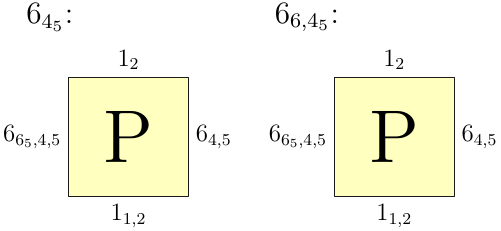}
            \caption{The way to glue the facets $6_{4_5}$ and $6_{6,4_5}$ of $L_7$.}\label{f19}
    \end{figure}    
    The same argument as before leads to the following: 
    \begin{align*}
            \quad \quad \quad 3_{4,4_5,2} \sim 3_{6,4,4_5,2}; \quad \quad
            & 3_{4_5,2} \sim 3_{6,4_5,2}; \quad \quad
            & 3_{6_5,4,4_5,2} \sim 3_{6_5,6,4,4_5,2}; \\ 
            3_{4,4_5} \sim 3_{6,4,4_5}; \quad \quad
            & 3_{1,4,4_5,2} \sim 3_{6,1,4,4_5,2}; \quad \quad
            & 3_{1,4_5,2} \sim 3_{6,1,4_5,2}; \\
            3_{6_5,1,4,4_5,2} \sim 3_{6_5,6,1,4,4_5,2}; \quad \quad
            & 3_{4_5} \sim 3_{6,4_5}; \quad \quad
            & 3_{6_5,4,4_5} \sim 3_{6_5,6,4,4_5}.
        \end{align*}


    Moreover we have:
    $$
            7 \cap_3 7_{6}; \quad \quad
             7_1 \cap_3 7_{6,1}; \quad \quad
             7_{6_5} \cap_3 7_{6_5,6}; \quad \quad
             7_{6_5,1} \cap_3 7_{6_5,6,1}.
    $$
    

    \end{itemize}

    We now consider the gluings which are specific for each one of the three cases.
    In each case we glue $1_2$ with $1_{1,2}$. We refer to Figure \ref{f12} and \ref{f20}.

    \begin{figure}\label{f20}
    
    \includegraphics{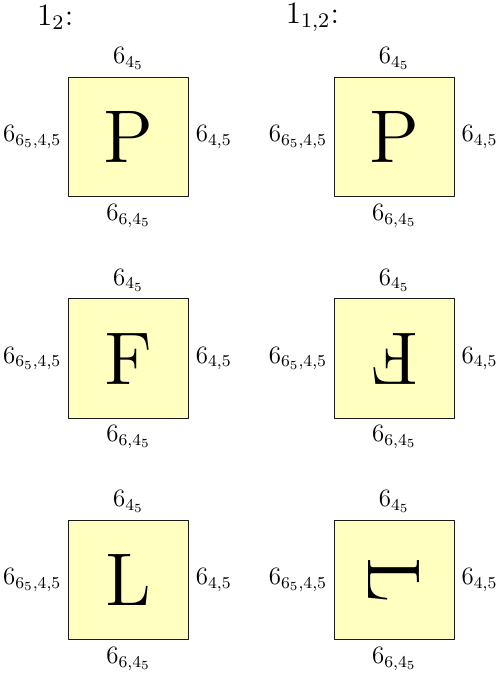}
            \caption{The way to glue the facets $1_{2}$ and $1_{1,2}$ of $L_7$ for the $3$-torus case (top), the $\frac{1}{2}$-twist manifold case (center) and the $\frac{1}{4}$-twist manifold case (bottom).}
    \end{figure}

    \begin{itemize}
        \item $3$-torus:
        \begin{align*}
            \quad \quad \quad 3_{6_5,4,4_5,2} \sim 3_{6_5,1,4,4_5,2}; \quad \quad
            & 3_{4_5,2} \sim 3_{1,4_5,2}; \quad \quad
            & 3_{4,4_5,2} \sim 3_{1,4,4_5,2}; \\
            3_{6_5,4,2} \sim 3_{6_5,1,4,2}; \quad \quad
            & 3_{2} \sim 3_{1,2}; \quad \quad
            & 3_{4,2} \sim 3_{1,4,2}; \\
            3_{6_5,6,4,4_5,2} \sim 3_{6_5,6,1,4,4_5,2}; \quad \quad
            & 3_{6,4_5,2} \sim 3_{6,1,4_5,2}; \quad \quad
            & 3_{6,4,4_5,2} \sim 3_{6,1,4,4_5,2}.
        \end{align*}


        Moreover we have:
        \begin{align*}
            7 \cap_2 7_{1}; \quad \quad
             7_6 \cap_2 7_{6,1}; \quad \quad
             7_{6_5} \cap_2 7_{6_5,1}; \quad \quad
             7_{6_5,6} \cap_2 7_{6_5,6,1}.
        \end{align*}


        Putting together the results of the three gluings for the torus, we have the thesis for $R_T$. We proceed similarly for the other cases.

        \item $\frac{1}{2}$-twist manifold:
        \begin{align*}
            \quad \quad \quad 3_{6_5,4,4_5,2} \sim 3_{6,1,4,4_5,2}; \quad \quad
            & 3_{4_5,2} \sim 3_{6,1,4_5,2}; \quad \quad
            & 3_{4,4_5,2} \sim 3_{6_5,6,1,4,4_5,2}; \\
            3_{6_5,4,2} \sim 3_{1,4,2}; \quad \quad
            & 3_{2} \sim 3_{1,2}; \quad \quad
            & 3_{4,2} \sim 3_{6_5,1,4,2}; \\
            3_{6_5,6,4,4_5,2} \sim 3_{1,4,4_5,2}; \quad \quad
            & 3_{6,4_5,2} \sim 3_{1,4_5,2}; \quad \quad
            & 3_{6,4,4_5,2} \sim 3_{6_5,1,4,4_5,2}.
        \end{align*}


        Moreover we have:
        $$
            7 \cap_2 7_{6_5,6,1}; \quad \quad
             7_1 \cap_2 7_{6_5,6}; \quad \quad
             7_{6_5} \cap_2 7_{6,1}; \quad \quad
             7_{6_5,1} \cap_2 7_{6}.
        $$
        


        \item $\frac{1}{4}$-twist manifold:
        \begin{align*}
            \quad \quad \quad 3_{6_5,4,4_5,2} \sim 3_{1,4,4_5,2}; \quad \quad
            & 3_{4_5,2} \sim 3_{1,4,2}; \quad \quad
            & 3_{4,4_5,2} \sim 3_{6,1,4,4_5,2}; \\
            3_{6_5,4,2} \sim 3_{1,4_5,2}; \quad \quad
            & 3_{2} \sim 3_{1,2}; \quad \quad
            & 3_{4,2} \sim 3_{6,1,4_5,2}; \\
            3_{6_5,6,4,4_5,2} \sim 3_{6_5,1,4,4_5,2}; \quad \quad
            & 3_{6,4_5,2} \sim 3_{6_5,1,4,2}; \quad \quad
            & 3_{6,4,4_5,2} \sim 3_{6_5,6,1,4,4_5,2}.
        \end{align*}


        Moreover we have:
        $$
         7 \cap_2 7_{6,1}; \quad \quad
             7_6 \cap_2 7_{6_5,6,1}; \quad \quad
             7_{6_5} \cap_2 7_{1}; \quad \quad
             7_{6_5,6} \cap_2 7_{6_5,1}.
        $$
    
    \end{itemize}
    
    In the last part of the proof we study the gluings of $P_8$ to form $R_{HW}$.

    \begin{itemize}
    \item $\underline{6_{4,5}}$: We refer to Figure \ref{f13.1} and \ref{f21}.
    \begin{figure}
    \centering
    \begin{minipage}{.5\textwidth}
    \centering
    \includegraphics{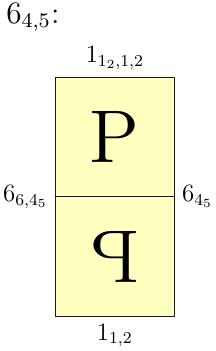}
        \caption{The way to glue the facet $6_{4,5}$ of $L_8$.}\label{f21}
    \end{minipage}%
    \begin{minipage}{.5\textwidth}
    \centering
    \includegraphics{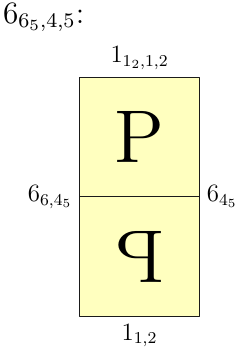}
        \caption{The way to glue the facet $6_{6_5,4,5}$ of $L_8$.}\label{f22}
    \end{minipage}
    \end{figure}
    \begin{align*}
            \quad \quad \quad 3_{1_2,6,1,4,4_5,2} \sim 3_{4,4_5,2}; \quad \quad
            & 3_{1_2,1,4,2} \sim 3_{4,2}; \quad \quad
            & 3_{1_2,1,4,4_5,2} \sim 3_{6,4,4_5,2}; \\
            3_{1_2,6,4,4_5} \sim 3_{4,4_5}; \quad \quad
            & 3_{1_2,4} \sim 3_{4}; \quad \quad
            & 3_{1_2,4,4_5} \sim 3_{6,4,4_5}; \\
            3_{6,4,4_5,2} \sim 3_{1,4,4_5,2}; \quad \quad
            & 3_{4,2} \sim 3_{1,4,2}; \quad \quad
            & 3_{4,4_5,2} \sim 3_{6,1,4,4_5,2}.
    \end{align*}

    Moreover we have:
    $$
           7_{1_2,6,1} \cap_3 7; \quad \quad
             7_{1_2,1} \cap_3 7_6; \quad \quad
             7_{1_2,6} \cap_3 7_1; \quad \quad
             7_{1_2} \cap_3 7_{6,1}.
    $$

    \item $\underline{6_{6_5,4,5}}$: We refer to Figure \ref{f13.2} and \ref{f22}.
    \begin{align*}
            \quad \quad \quad 3_{1_2,6_5,6,1,4,4_5,2} \sim 3_{6_5,4,4_5,2}; \quad \quad
            & 3_{1_2,6_5,1,4,2} \sim 3_{6_5,4,2}; \quad \quad
            & 3_{1_2,6_5,1,4,4_5,2} \sim 3_{6_5,6,4,4_5,2}; \\
            3_{1_2,6_5,6,4,4_5} \sim 3_{6_5,4,4_5}; \quad \quad
            & 3_{1_2,6_5,4} \sim 3_{6_5,4}; \quad \quad
            & 3_{1_2,6_5,4,4_5} \sim 3_{6_5,6,4,4_5}; \\
            3_{6_5,6,4,4_5,2} \sim 3_{6_5,1,4,4_5,2}; \quad \quad
            & 3_{6_5,4,2} \sim 3_{6_5,1,4,2}; \quad \quad
            & 3_{6_5,4,4_5,2} \sim 3_{6_5,6,1,4,4_5,2}.
    \end{align*}

    Moreover we have:
    $$
           7_{1_2,6_5,6,1} \cap_3 7_{6_5}; \quad \quad
             7_{1_2,6_5,1} \cap_3 7_{6_5,6}; \quad \quad
             7_{1_2,6_5,6} \cap_3 7_{6_5,1}; \quad \quad
             7_{1_2,6_5} \cap_3 7_{6_5,6,1}.
    $$

    \item \underline{$6_{4_5}$ and $6_{6,4_5}$}: We refer to Figure \ref{f13.3}, \ref{f13.4} and \ref{f23}.
    \begin{figure}
    \includegraphics{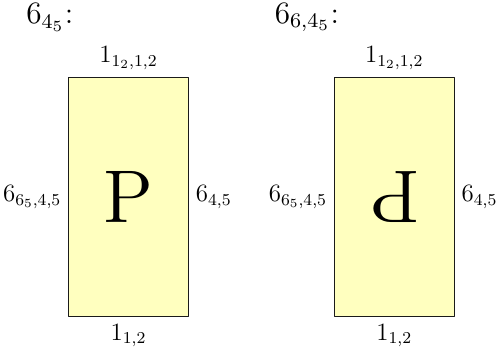}
        \caption{The way to glue the facets $6_{4_5}$ and $6_{6,4_5}$ of $L_8$.}\label{f23}
    \end{figure}
    \begin{align*}
            \quad \quad \quad 3_{1_2,6_5,1,4,4_5,2} \sim 3_{6,1,4,4_5,2}; \quad \quad
            & 3_{1_2,1,4_5,2} \sim 3_{6,1,4_5,2}; \quad \quad
            & 3_{1_2,1,4,4_5,2} \sim 3_{6_5,6,1,4,4_5,2}; \\
            3_{1_2,6_5,4,4_5} \sim 3_{6,4,4_5}; \quad \quad
            & 3_{1_2,4_5} \sim 3_{6,4_5}; \quad \quad
            & 3_{1_2,4,4_5} \sim 3_{6_5,6,4,4_5}; \\
            3_{6_5,4,4_5,2} \sim 3_{6,4,4_5,2}; \quad \quad
            & 3_{4_5,2} \sim 3_{6,4_5,2}; \quad \quad
            & 3_{4,4_5,2} \sim 3_{6_5,6,4,4_5,2} \\
            3_{6_5,4,4_5} \sim 3_{1_2,6,4,4_5}; \quad \quad
            & 3_{4_5} \sim 3_{1_2,6,4_5}; \quad \quad
            & 3_{4,4_5} \sim 3_{1_2,6_5,6,4,4_5} \\
            3_{6_5,1,4,4_5,2} \sim 3_{1_2,6,1,4,4_5,2}; \quad \quad
            & 3_{1,4_5,2} \sim 3_{1_2,6,1,4_5,2}; \quad \quad
            & 3_{1,4,4_5,2} \sim 3_{1_2,6_5,6,1,4,4_5,2}.
    \end{align*}

    Moreover we have:
    \begin{align*}
           7_{1_2,6_5,1} \cap_3 7_{6,1}; \quad \quad
            & 7_{1_2,1} \cap_3 7_{6_5,6,1}; \quad \quad
            & 7_{1_2,6_5} \cap_3 7_{6}; \quad \quad
            & 7_{1_2} \cap_3 7_{6_5,6} \\
            7_{6_5} \cap_3 7_{1_2,6}; \quad \quad
            & 7 \cap_3 7_{1_2,6_5,6}; \quad \quad
            & 7_{6_5,1} \cap_3 7_{1_2,6,1}; \quad \quad
            & 7_{1} \cap_3 7_{1_2,6_5,6,1}.
    \end{align*}

    \item \underline{$1_{1,2}$ and $1_{1_2,1,2}$}: We refer to Figure \ref{f13} and \ref{f24}.
    \begin{figure}
    \includegraphics{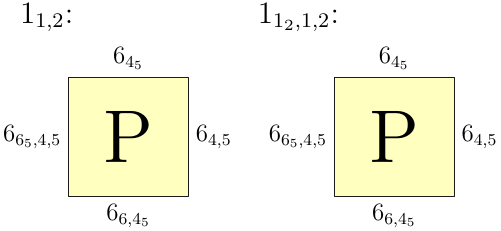}
        \caption{The way to glue the facets $1_{1,2}$ and $1_{1_2,1,2}$ of $L_8$.}\label{f24}
    \end{figure}
     \begin{align*}
            \quad \quad \quad 3_{6_5,1,4,4_5,2} \sim 3_{1_2,6_5,1,4,4_5,2}; \quad \quad
            & 3_{1,4_5,2} \sim 3_{1_2,1,4_5,2}; \quad \quad
            & 3_{1,4,4_5,2} \sim 3_{1_2,1,4,4_5,2}; \\
            3_{6_5,1,4,2} \sim 3_{1_2,6_5,1,4,2}; \quad \quad
            & 3_{1,2} \sim 3_{1_2,1,2}; \quad \quad
            & 3_{1,4,2} \sim 3_{1_2,1,4,2}; \\
            3_{6_5,6,1,4,4_5,2} \sim 3_{1_2,6_5,6,1,4,4_5,2}; \quad \quad
            & 3_{6,1,4_5,2} \sim 3_{1_2,6,1,4_5,2}; \quad \quad
            & 3_{6,1,4,4_5,2} \sim 3_{1_2,6,1,4,4_5,2}.
    \end{align*}

    Moreover we have:
    $$
           7_{6_5,1} \cap_2 7_{1_2,6_5,1}; \quad \quad
             7_1 \cap_2 7_{1_2,1}; \quad \quad
             7_{6_5,6,1} \cap_2 7_{1_2,6_5,6,1}; \quad \quad
             7_{6,1} \cap_2 7_{1_2,6,1}.
    $$

    \end{itemize}
\end{proof}

\subsection{The space \texorpdfstring{$R$}{R} is a \texorpdfstring{$1$}{1}-cusped developable reflectofold.}\label{ref3}
We conclude here the proof of Theorem \ref{Rbello}.

Recall Definition \ref{defcgraph} of the corner graphs $G_3$ and $G_7$. We can now recover enough information about them.

\begin{defn}
Let $\widetilde{G_3}$ be the graph obtained by identifying the vertices of the adjacency graph of facets of type $3$ of $P$ by the relation $\sim$.
Let $\widetilde{G_7}$ be the graph obtained by taking the adjacency graph of facets of type $7$ of $P$ and adding a labelled edge $(F,G;k)$, for every $F\cap_k G$ in Lemma \ref{facce}.    
\end{defn}

\begin{prop}
\label{G3}
    The corner graph $G_3$ is a subgraph of $\widetilde{G_3}$. More specifically, the vertices of the two graphs are the same, while if two vertices of $\widetilde{G_3}$ have $m$ edges connecting them with label $l$, in $G_3$ we have $n$ edges with label $l$ between the corresponding vertices, with $1\le n \le m$.
\end{prop}
\begin{proof}
    The vertices of $\widetilde{G_3}$ coincide with the ones of $G_3$ by Lemma \ref{facce}.

    Let $\overline{3_X}=\{3_{X_1},\dots,3_{X_k}\}$ and $\overline{3_Y}=\{3_{Y_1},\dots,3_{Y_k}\}$ be two equivalence classes of type-$3$ facets of $P$.
    By construction, for every ridge between two facets $3_{X_i}$ and $3_{Y_j}$ with dihedral angle $\frac{\pi}{k}$, there is an edge in $\widetilde{G_3}$ between the vertices $\overline{3_X}$ and $\overline{3_Y}$ with label $k$.

    By Lemma \ref{corner} a type-$(3,3)$ corner of $R$ is the image through $p$ of a union of some type-$(3,3)$ ridges of $P$. Hence if the image through $p$ of the union of $r$ ridges is a single corner between the facets $\overline{3_X}$ and $\overline{3_Y}$ in $R$, then in $G_3$ we have one edge between $\overline{3_X}$ and $\overline{3_Y}$; while in $\widetilde{G_3}$ we have $r$ edges between them. It is easy to check that these $r$ edges have the same label associated (by checking the adiacency graph of $P$ in Table \ref{t7} and \ref{t8}, and the results of Lemma \ref{facce}).
\end{proof}

\begin{prop}
\label{G7}
    The corner graph $G_7$ is equal to $\widetilde{G_7}$.
\end{prop}
\begin{proof}
    The vertices of $\widetilde{G_7}$ coincide with the ones of $G_7$ by Lemma \ref{tipi}.
    By Lemma \ref{corner} the edges of $G_7$ are the ones of $\widetilde{G_7}$. 
\end{proof}

It is easy to verify (by checking the adjacency matrices of $P$ in Table \ref{t7} and \ref{t8}, and the results of Lemma \ref{facce}) that $\widetilde{G_3}$ and $\widetilde{G_7}$ have no loop (an edge connecting one vertex to itself) and if two vertices have more then one edge connecting them, all these edges have the same label. Hence it makes sense to define the adjacency matrices of $R$.

\begin{defn}
    For $i=3,7$, the \emph{type-$i$ adjacency matrix} of $R$ is the matrix where in the entry corresponding to the type-$i$ facet $A$ and $B$ we put $1$ if $A=B$, we put $0$ if $A\cap B=\emptyset$, we put $k$ if the dihedral angle at the corners of $A\cap B$ is $\frac{\pi}{k}$ and we put $\underline{\alpha}$ if the dihedral angle at the corners of $A\cap B$ is $\alpha\ne\frac{\pi}{k}$, for every $k$.
\end{defn}

One could also obtain the adjacency matrix of $R$, but we are only interested in the type-$3$ and type-$7$ ones, which are the submatrices corresponding to the facets of type $3$ and of type $7$, respectively.


\begin{prop}
    For $i=3,7$, the type-$i$ adjacency matrix of $R$ is in Tables \ref{tr1}, \ref{tr2}, \ref{tr3} and \ref{tr4}.
\end{prop}
\begin{proof}
    By Proposition \ref{G3} and \ref{G7}, if $\widetilde{G_i}$ has at least one edge with label $l$ between the vertices $A$ and $B$, then in the matrix the entry between $A$ and $B$ is $l$. If in $\widetilde{G_i}$ there is no edge between $A$ and $B$, then there is a $0$ in the corresponding entry.
\end{proof}



\begin{table}
\begin{tabular}{|c||c|c|c|c|c|c|c|c|}
 \hhline{|=#=|=|=|=|=|=|=|=|}
$3$ & $1$ & $2$ & $3$ &  & $3$ &  &  & \\ \hline
$3_2$ & $2$ & $1$ &  & $3$ &  & $3$ &  & \\ \hline
$3_{4_5}$ & $3$ &  & $1$ & $2$ &  &  & $3$ & \\ \hline
$3_{4_5,2}$ &  & $3$ & $2$ & $1$ &  &  &  & $3$ \\ \hline
$3_{4}$ & $3$ &  &  &  & $1$ & $2$ & $3$ & \\ \hline
$3_{4,2}$ &  & $3$ &  &  & $2$ & $1$ &  & $3$ \\ \hline
$3_{4,4_5}$ &  &  & $3$ &  & $3$ &  & $1$ & $2$ \\ \hline
$3_{4,4_5,2}$ &  &  &  & $3$ &  & $3$ & $2$ & $1$\\ \hline
\end{tabular} \hspace{20pt}
\begin{tabular}{|c||c|c|c|c|c|c|c|c|}
 \hhline{|=#=|=|=|=|=|=|=|=|}
$7$ & $1$ & $2$ & $3$ &  & $3$ &  &  & \\ \hline
$7_1$ & $2$ & $1$ &  & $3$ &  & $3$ &  & \\ \hline
$7_{6}$ & $3$ &  & $1$ & $2$ &  &  & $3$ & \\ \hline
$7_{6,1}$ &  & $3$ & $2$ & $1$ &  &  &  & $3$ \\ \hline
$7_{6_5}$ & $3$ &  &  &  & $1$ & $2$ & $3$ & \\ \hline
$7_{6_5,1}$ &  & $3$ &  &  & $2$ & $1$ &  & $3$ \\ \hline
$7_{6_5,6}$ &  &  & $3$ &  & $3$ &  & $1$ & $2$ \\ \hline
$7_{6_5,6,1}$ &  &  &  & $3$ &  & $3$ & $2$ & $1$\\ \hline
\end{tabular}
\caption{Type-$3$ and type-$7$ adjacency matrices of $R_T$.}\label{tr1}
\end{table}

\begin{table}
\begin{tabular}{|c||c|c|c|c|c|c|c|c|}
 \hhline{|=#=|=|=|=|=|=|=|=|}
$3$ & $1$ & $2$ & $3$ &  & $3$ &  &  & \\ \hline
$3_2$ & $2$ & $1$ &  & $3$ &  & $3$ &  & \\ \hline
$3_{4_5}$ & $3$ &  & $1$ & $2$ &  &  & $3$ & \\ \hline
$3_{4_5,2}$ &  & $3$ & $2$ & $1$ &  &  &  & $3$ \\ \hline
$3_{4}$ & $3$ &  &  &  & $1$ & $2$ & $3$ & \\ \hline
$3_{4,2}$ &  & $3$ &  &  & $2$ & $1$ &  & $3$ \\ \hline
$3_{4,4_5}$ &  &  & $3$ &  & $3$ &  & $1$ & $2$ \\ \hline
$3_{4,4_5,2}$ &  &  &  & $3$ &  & $3$ & $2$ & $1$\\ \hline
\end{tabular} \hspace{20pt}
\begin{tabular}{|c||c|c|c|c|c|c|c|c|}
 \hhline{|=#=|=|=|=|=|=|=|=|}
$7$ & $1$ & $2$ & $3$ &  & $3$ &  &  & $2$ \\ \hline
$7_1$ & $2$ & $1$ &  & $3$ &  & $3$ & $2$ & \\ \hline
$7_{6}$ & $3$ &  & $1$ & $2$ &  & $2$ & $3$ & \\ \hline
$7_{6,1}$ &  & $3$ & $2$ & $1$ & $2$ &  &  & $3$ \\ \hline
$7_{6_5}$ & $3$ &  &  & $2$ & $1$ & $2$ & $3$ & \\ \hline
$7_{6_5,1}$ &  & $3$ & $2$ &  & $2$ & $1$ &  & $3$ \\ \hline
$7_{6_5,6}$ &  & $2$ & $3$ &  & $3$ &  & $1$ & $2$ \\ \hline
$7_{6_5,6,1}$ & $2$ &  &  & $3$ &  & $3$ & $2$ & $1$\\ \hline
\end{tabular}
\caption{Type-$3$ and type-$7$ adjacency matrices of $R_{\frac{1}{2}}$.}\label{tr2}
\end{table}

\begin{table}
\begin{tabular}{|c||c|c|c|c|c|c|c|c|}
 \hhline{|=#=|=|=|=|=|=|=|=|}
$3$ & $1$ & $2$ & $3$ &  & $3$ &  &  & \\ \hline
$3_2$ & $2$ & $1$ &  & $3$ &  & $3$ &  & \\ \hline
$3_{4_5}$ & $3$ &  & $1$ & $2$ &  & $2$ & $3$ & \\ \hline
$3_{4_5,2}$ &  & $3$ & $2$ & $1$ & $2$ &  &  & $3$ \\ \hline
$3_{4}$ & $3$ &  &  & $2$ & $1$ & $2$ & $3$ & \\ \hline
$3_{4,2}$ &  & $3$ & $2$ &  & $2$ & $1$ &  & $3$ \\ \hline
$3_{4,4_5}$ &  &  & $3$ &  & $3$ &  & $1$ & $2$ \\ \hline
$3_{4,4_5,2}$ &  &  &  & $3$ &  & $3$ & $2$ & $1$\\ \hline
\end{tabular} \hspace{20pt}
\begin{tabular}{|c||c|c|c|c|c|c|c|c|}
 \hhline{|=#=|=|=|=|=|=|=|=|}
$7$ & $1$ & $2$ & $3$ & $2$ & $3$ &  &  & \\ \hline
$7_1$ & $2$ & $1$ &  & $3$ & $2$ & $3$ &  & \\ \hline
$7_{6}$ & $3$ &  & $1$ & $2$ &  &  & $3$ & $2$ \\ \hline
$7_{6,1}$ & $2$ & $3$ & $2$ & $1$ &  &  &  & $3$ \\ \hline
$7_{6_5}$ & $3$ & $2$ &  &  & $1$ & $2$ & $3$ & \\ \hline
$7_{6_5,1}$ &  & $3$ &  &  & $2$ & $1$ & $2$ & $3$ \\ \hline
$7_{6_5,6}$ &  &  & $3$ &  & $3$ & $2$ & $1$ & $2$ \\ \hline
$7_{6_5,6,1}$ &  &  & $2$ & $3$ &  & $3$ & $2$ & $1$\\ \hline
\end{tabular}
\caption{Type-$3$ and type-$7$ adjacency matrices of $R_{\frac{1}{4}}$.}\label{tr3}
\end{table}

\begin{table}
\begin{tabular}{|c||c|c|c|c|c|c|c|c|c|c|c|c|c|c|c|c|}
\hhline{|=#=|=|=|=|=|=|=|=|=|=|=|=|=|=|=|=|}
$3$ & $1$ & $2$  & $3$  &  & $2$  &  & $3$ & &  & $3$ & $3$ & & & & & \\ \hline
$3_2$ & $2$  & $1$ &  & $2$  &   & $3$ &  & $3$ & & & & $3$ & & & & \\ \hline
$3_4$ & $3$ &  & $1$ & $3$ &  & $2$ &  & & $3$ & & $3$ & & & & & \\ \hline
$3_{1_2}$ &  & $2$ & $3$ & $1$ & $2$ &  & $3$ &  & & $3$ & $3$ & & & & & \\ \hline
$3_{1,2}$ & $2$  &  &  & $2$ & $1$ & $3$ &  & & & & & $3$ &  & $3$ & & \\ \hline
$3_{4,2}$ &  & $3$  & $2$ &  & $3$ & $1$ &  &  & & & $2$ & & $3$ & & & $3$ \\ \hline
$3_{4_5}$ & $3$ &  &  & $3$ &  &  & $1$ & $2$ & $3$ & & & & & $2$ &  & \\ \hline
$3_{4_5,2}$ &  & $3$ &  &  &  &  & $2$ & $1$ & & $2$ & & & & & $3$ & $3$ \\ \hline
$3_{4,4_5}$ &  & & $3$ & & & & $3$ & & $1$ &  & $3$ & &  & & $2$ & $2$ \\ \hline
$3_{6,4_5}$ & $3$ & & & $3$ & & & & $2$ &  & $1$ & &  & $3$ & $2$ & & \\ \hline
$3_{6_5,4}$ & $3$ & & $3$ & $3$ & & $2$ & & & $3$ & & $1$ & $2$ & & &  & \\ \hline
$3_{6_5,4,2}$ & & $3$ & & & $3$ & & & & &  & $2$ & $1$ & & & $3$ & $3$ \\ \hline
$3_{6,4,4_5}$ & & & & &  & $3$ & & &  & $3$ & & & $1$ &  & $2$ & $2$ \\ \hline
$3_{1,4_5,2}$ & & & & & $3$ & & $2$ & & & $2$ & & &  & $1$ & $3$ & $3$ \\ \hline
$3_{4.4_5,2}$ & & & & & & &  & $3$ & $2$ & &  & $3$ & $2$ & $3$ & $1$ & $3$  \\ \hline
$3_{1,4,4_5,2}$ & & & & & & $3$ & & $3$ & $2$ & & & $3$ & $2$ & $3$ & $3$ & $1$ \\ \hline
\end{tabular}

\vspace{20pt}

\begin{tabular}{|c||c|c|c|c|c|c|c|c|c|c|c|c|c|c|c|c|}
\hhline{|=#=|=|=|=|=|=|=|=|=|=|=|=|=|=|=|=|}
$7$ & $1$ & $2$ & $3$ &  & $3$ &  &  & & $2$ & & & $3$ & & & $3$ & \\ \hline
$7_1$ & $2$ & $1$ &  & $3$ &  & $3$ &  & & & $2$ & $3$ & & & & & $3$ \\ \hline
$7_{6}$ & $3$ &  & $1$ & $2$ &  &  & $3$ & & & $3$ & $2$ & & $3$ & & & \\ \hline
$7_{6,1}$ &  & $3$ & $2$ & $1$ &  &  &  & $3$ & $3$ & & & $2$ & & $3$ & & \\ \hline
$7_{6_5}$ & $3$ &  &  &  & $1$ & $2$ & $3$ & & & & $3$ & & $2$ & & & $3$ \\ \hline
$7_{6_5,1}$ &  & $3$ &  &  & $2$ & $1$ &  & $3$ & & & & $3$ & & $2$ & $3$ & \\ \hline
$7_{6_5,6}$ &  &  & $3$ &  & $3$ &  & $1$ & $2$ & $3$ & & & & & $3$ & $2$ & \\ \hline
$7_{6_5,6,1}$ &  &  &  & $3$ &  & $3$ & $2$ & $1$ & & $3$ & & & $3$ & & & $2$ \\ \hline
$7_{1_2}$ & $2$ & & & $3$ & & & $3$ & & $1$ & $2$ & $3$ & & $3$ & & & \\ \hline
$7_{1_2,1}$ & & $2$ & $3$ & & & & & $3$ & $2$ & $1$ & & $3$ & & $3$ & & \\ \hline
$7_{1_2,6}$ & & $3$ & $2$ & & $3$ & & & & $3$ & & $1$ & $2$ & & & $3$ & \\ \hline
$7_{1_2,6,1}$ & $3$ & & & $2$ & & $3$ & & & & $3$ & $2$ & $1$ & & & & $3$ \\ \hline
$7_{1_2,6_5}$ & & & $3$ & & $2$ & & & $3$ & $3$ & & & & $1$ & $2$ & $3$ & \\ \hline
$7_{1_2,6_5,1}$ & & & & $3$ & & $2$ & $3$ & & & $3$ & & & $2$ & $1$ & & $3$ \\ \hline
$7_{1_2,6_5,6}$ & $3$ & & & & & $3$ & $2$ & & & & $3$ & & $3$ & & $1$ & $2$ \\ \hline
$7_{1_2,6_5,6,1}$ & & $3$ & & & $3$ & & & $2$ & & & & $3$ & & $3$ & $2$ & $1$ \\ \hline
\end{tabular}
\caption{Type-$3$ and type-$7$ adjacency matrices of $R_{HW}$.}\label{tr4}
\end{table}

\begin{prop}
\label{loc}
    The space $R$ is a finite-volume reflectofold.
\end{prop}
\begin{proof}
    The dihedral angles at the type-$(3,3)$ and type-$(7,7)$ corners of $R$ are all of the form $\frac{\pi}{k}$ since in Tables \ref{tr1}, \ref{tr2}, \ref{tr3} and \ref{tr4} there are not underlined labels.
    Every type-$(3,7)$ ridge of $P$ has dihedral angle $\frac{\pi}{2}$ by Proposition \ref{angolo} (since the dihedral angle between $\textbf{3}$ and $\textbf{7}$ in $P_0$ is $\frac{\pi}{2}$). Hence, by Lemma \ref{corner}, also every type-$(3,7)$ corner of $R$ has dihedral angle $\frac{\pi}{2}$. By Lemma \ref{corner}, this runs out all the corners of $R$.

    We show that $R$ is locally a Coxeter polytope. The faces of $P$ induce a natural stratification of $R$ in closed strata. We have that $R$ is locally modelled on $\matH^n$ near the non-compact strata and far from the compact strata, since its end is isometric to a cusp (with section a flat, closed manifold) by construction. We have that $R$ is locally a Coxeter polytope near the compact strata since we have proved that the angle corresponding to the corners are in the form $\frac{\pi}{k}$.
    
    Moreover, $R$ is complete by construction, since we glued using reflections through copies of the facets of $P_0$. Hence $R$ is a reflectofold.
    Finally, the polytope $P$ is tessellated into a finite number of copies of $P_0$, which has finite volume, hence also $R$ has finite volume.
\end{proof}

\begin{prop}
\label{cusp}
    The reflectofold $R$ is $1$-cusped, has compact, non-empty boundary and is orientable. Moreover, the cusp of $R_T, R_{\frac{1}{2}},R_{\frac{1}{4}}, R_{HW}$ has section, the $3$-torus, the $\frac{1}{2}$-twist manifold, the $\frac{1}{4}$-twist manifold, the Hantzsche-Wendt manifold, respectively.
\end{prop}
\begin{proof}
    The boundary of $R$ is the image of the union of the facets of $P$ that we do not glue. Since these facets are of type $3$ or $7$, that are compact, the boundary of $R$ is compact (and non-empty).

    Since by construction we glued $P$ in a way that this induces a gluing of the link $L_7$ (of the only ideal vertex) to form the $3$-torus, the $\frac{1}{2}$-twist manifold, the $\frac{1}{4}$-twist manifold, the space $R$ has exaclty one cusp with the requested section.

    Finally, the space $R$ is orientable, since it is homeomorphic to $E\times [0,1)$, where $E$ is the cusp section, which is orientable.
\end{proof}

\begin{prop}
\label{proprieta}
    The reflectofold $R$ is developable.
\end{prop}
\begin{proof}
    Since the graph $\widetilde{G_i}$ has no loops, also the corner graph $G_i$ has no loops, for $i=3,7$, by Proposition \ref{G3} and Proposition \ref{G7}. Hence $R$ satisfies \ref{EF}.

    Moreover, $R$ satisfies \ref{AC}:
    \begin{itemize}
        \item If two facets $F$ and $G$ both of type $3$ (or $7$) intersect, then the dihedral angles of all the corners in $F \cap G$ coincide; indeed, as already stated, in $\widetilde{G_3}$ (or $\widetilde{G_7}$), and hence in $G_3$ (or $G_7$), if two vertices have more then one edge connecting them, all these edges have the same label.
        \item Every type-$(3,7)$ ridge of $P$ has dihedral angle $\frac{\pi}{2}$ by Proposition \ref{angolo} (since the dihedral angle between $\textbf{3}$ and $\textbf{7}$ in $P_0$ is $\frac{\pi}{2}$). Hence, by Lemma \ref{corner}, also every type-$(3,7)$ corner of $R$ has dihedral angle $\frac{\pi}{2}$.
    \end{itemize}
\end{proof}

Putting together Proposition \ref{cusp}, Proposition \ref{loc} and Proposition \ref{proprieta}, we have proved Theorem \ref{Rbello}.
Putting together Theorem \ref{Rbello} and Corollary \ref{maincor}, we have proved Theorem \ref{thm:main}.





\section*{Tables and Figures}\label{appx}
For reasons of space, we collect here the information \ref{item:I2} on $P_n$ for $n=0,\dots,8$, and the information \ref{item:I3} on $P_n$ for $n=2,\dots,8$.


\begin{table}
\centering
\begin{minipage}{.5\textwidth}
\centering
\begin{tabular}{|c||c|}
\hhline{|=#=|}
 $3$ & $1$ \\ \hline
\end{tabular} \hspace{20pt}
\begin{tabular}{|c||c|}
\hhline{|=#=|}
 $7$ & $1$ \\ \hline
\end{tabular}
\caption{Type-$3$ and type-$7$ adjacency matrices of $P_0$.} \label{t0}
\end{minipage}%
\begin{minipage}{.5\textwidth}
\centering
\begin{tabular}{|c||c|}
\hhline{|=#=|}
 $3$ & $1$ \\ \hline
\end{tabular} \hspace{20pt}
\begin{tabular}{|c||c|}
\hhline{|=#=|}
 $7$ & $1$ \\ \hline
\end{tabular}
\caption{Type-$3$ and type-$7$ adjacency matrices of $P_1$.} \label{t1}
\end{minipage}
\end{table}

\begin{figure}
\includegraphics{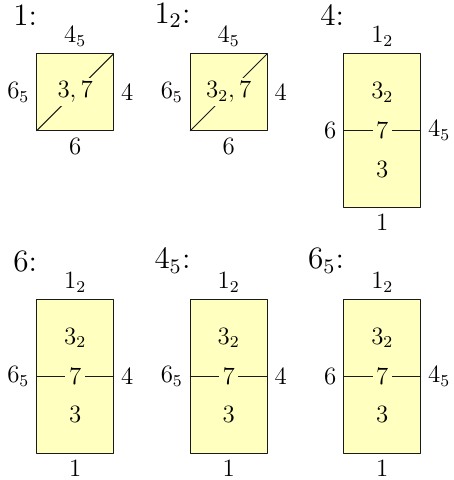}
    \caption{The facets of $L_2$.}\label{f7}
\end{figure}

\begin{table}
\centering
\begin{minipage}{.5\textwidth}
\centering
\begin{tabular}{|c||c|c|}
\hhline{|=#=|=|}
 $3$ & $1$ & $2$\\ \hline
 $3_2$ & $2$ & $1$\\ \hline
\end{tabular} \hspace{20pt}
\begin{tabular}{|c||c|}
\hhline{|=#=|}
 $7$ & $1$ \\ \hline
\end{tabular}
\caption{Type-$3$ and type-$7$ adjacency matrices of $P_2$.}\label{t2}
\end{minipage}%
\begin{minipage}{.5\textwidth}
\centering
\begin{tabular}{|c||c|c|c|c|}
\hhline{|=#=|=|=|=|}
$3$ & $1$ & $2$ & $3$ & \\ \hline
$3_2$ & $2$ & $1$ &  & $3$\\ \hline
$3_{4_5}$ & $3$ &  & $1$ & $2$\\ \hline
$3_{4_5,2}$ &  & $3$ & $2$ & $1$\\ \hline
\end{tabular} \hspace{20pt}
\begin{tabular}{|c||c|}
\hhline{|=#=|}
 $7$ & $1$ \\ \hline
\end{tabular}
\caption{Type-$3$ and type-$7$ adjacency matrices of $P_3$.}\label{t3}
\end{minipage}
\end{table}

\begin{figure}
\includegraphics{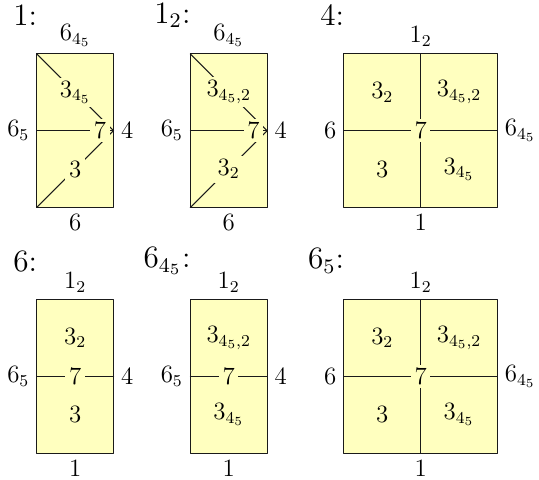}
    \caption{The facets of $L_3$}\label{f8}
\end{figure}

\begin{table}
\begin{tabular}{|c||c|c|c|c|c|c|c|c|}
\hhline{|=#=|=|=|=|=|=|=|=|}
$3$ & $1$ & $2$ & $3$ &  & $3$ &  &  & \\ \hline
$3_2$ & $2$ & $1$ &  & $3$ &  & $3$ &  & \\ \hline
$3_{4_5}$ & $3$ &  & $1$ & $2$ &  &  & $3$ & \\ \hline
$3_{4_5,2}$ &  & $3$ & $2$ & $1$ &  &  &  & $3$\\ \hline
$3_4$ & $3$ &  &  &  & $1$ & $2$ & $3$ & \\ \hline
$3_{4,2}$ &  & $3$ &  &  & $2$ & $1$ &  & $3$\\ \hline
$3_{4,4_5}$ &  &  & $3$ &  & $3$ &  & $1$ & $2$\\ \hline
$3_{4,4_5,2}$ &  &  &  & $3$ &  & $3$ & $2$ & $1$\\ \hline
\end{tabular} \hspace{20pt}
\begin{tabular}{|c||c|}
\hhline{|=#=|}
 $7$ & $1$ \\ \hline
\end{tabular}
\caption{Type-$3$ and type-$7$ adjacency matrices of $P_4$.}\label{t4}
\end{table}

\begin{figure}
\includegraphics{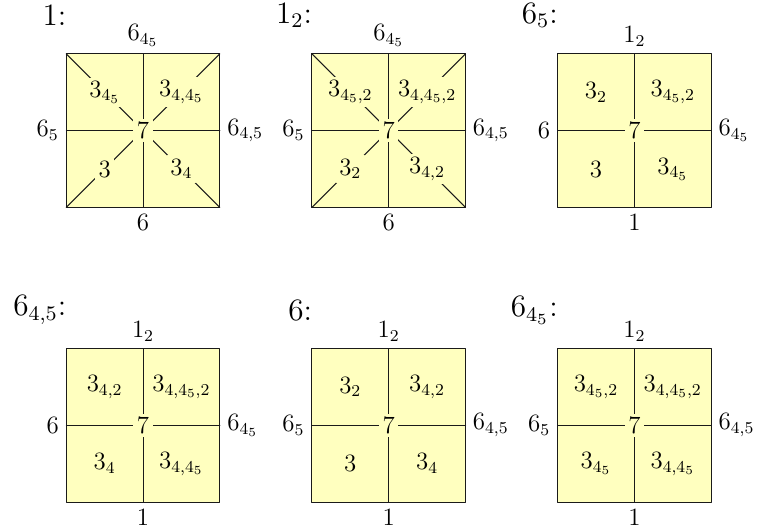}
    \caption{The facets of $L_4$.}\label{f9}
\end{figure}

\begin{table}
\begin{tabular}{|c||c|c|c|c|c|c|c|c|c|c|c|c|}
\hhline{|=#=|=|=|=|=|=|=|=|=|=|=|=|}
$3$ & $1$ & $2$ & $3$ &  & $3$ &  &  &  & $3$ &  &  & \\ \hline
$3_2$ & $2$ & $1$ &  & $3$ &  & $3$ &  &  &  &  &  & \\ \hline
$3_{4_5}$ & $3$ &  & $1$ & $2$ &  &  & $3$ &  &  & $3$ &  & \\ \hline
$3_{4_5,2}$ &  & $3$ & $2$ & $1$ &  &  &  & $3$ &  &  &  & \\ \hline
$3_4$ & $3$ &  &  &  & $1$ & $2$ & $3$ &  &  &  & $2$ & \\ \hline
$3_{4,2}$ &  & $3$ &  &  & $2$ & $1$ &  & $3$ &  &  &  & \\ \hline
$3_{4,4_5}$ &  &  & $3$ &  & $3$ &  & $1$ & $2$ &  &  &  & $2$\\ \hline
$3_{4,4_5,2}$ &  &  &  & $3$ &  & $3$ & $2$ & $1$ &  &  &  & \\ \hline
$3_{1,2}$ & $3$ &  &  &  &  &  &  &  & $1$ & $3$ & $3$ & \\ \hline
$3_{1,4_5,2}$ &  &  & $3$ &  &  &  &  &  & $3$ & $1$ &  & $3$\\ \hline
$3_{1,4,2}$ &  &  &  &  & $2$ &  &  &  & $3$ &  & $1$ & $3$\\ \hline
$3_{1,4,4_5,2}$ &  &  &  &  &  &  & $2$ &  &  & $3$ & $3$ & $1$\\ \hline
\end{tabular} \hspace{20pt}
\begin{tabular}{|c||c|c|}
 \hhline{|=#=|=|}
 $7$ & $1$ & $2$\\ \hline
 $7_1$ & $2$ & $1$\\ \hline
\end{tabular}
\caption{Type-$3$ and type-$7$ adjacency matrices of $P_5$.}\label{t5}
\end{table}

\begin{figure}
\includegraphics{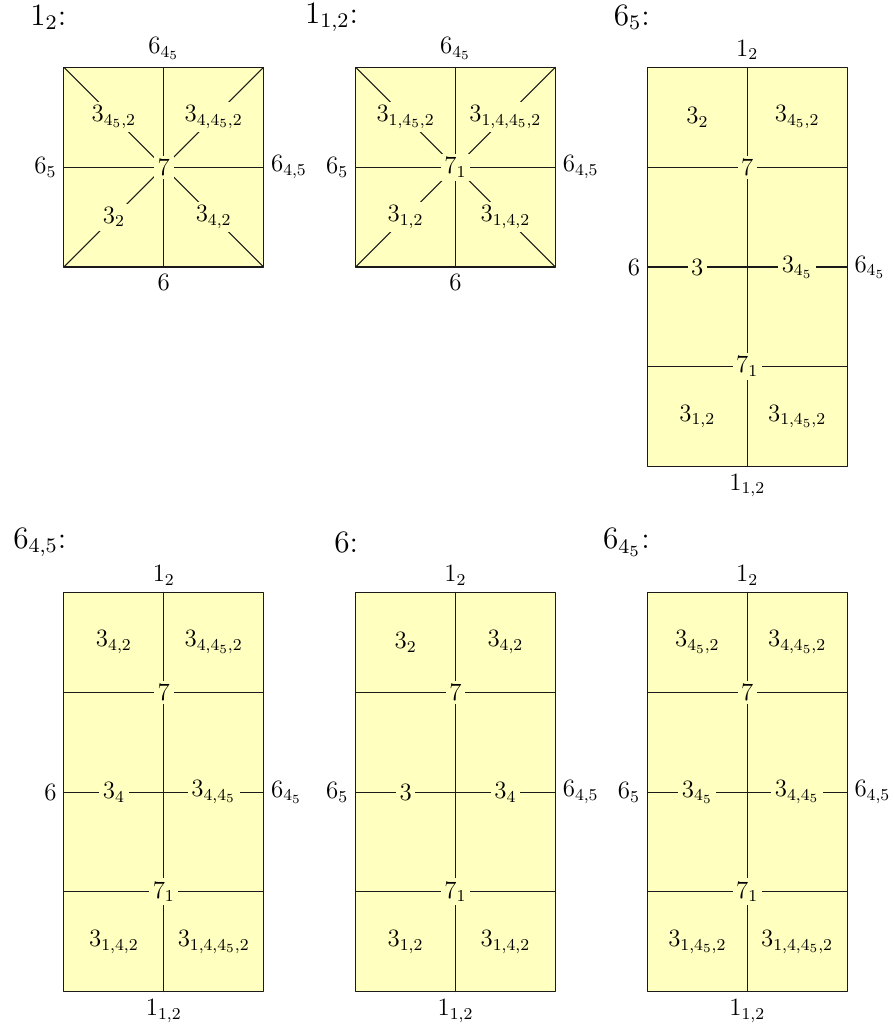}
    \caption{The facets of $L_5$.}\label{f10}
\end{figure}

\begin{table}
\begin{tabular}{|c||c|c|c|c|c|c|c|c|c|c|c|c|c|c|c|c|c|c|}
\hhline{|=#=|=|=|=|=|=|=|=|=|=|=|=|=|=|=|=|=|=|}
$3$ & $1$ & $2$ & $3$ &  & $3$ &  &  &  & $3$ &  &  &  & $3$ &  &  &  &  & \\ \hline
$3_2$ & $2$ & $1$ &  & $3$ &  & $3$ &  &  &  &  &  &  &  & $3$ &  &  &  & \\ \hline
$3_{4_5}$ & $3$ &  & $1$ & $2$ &  &  & $3$ &  &  & $3$ &  &  &  &  &  &  &  & \\ \hline
$3_{4_5,2}$ &  & $3$ & $2$ & $1$ &  &  &  & $3$ &  &  &  &  &  &  &  &  &  & \\ \hline
$3_4$ & $3$ &  &  &  & $1$ & $2$ & $3$ &  &  &  & $2$ &  &  &  & $3$ &  &  & \\ \hline
$3_{4,2}$ &  & $3$ &  &  & $2$ & $1$ &  & $3$ &  &  &  &  &  &  &  & $3$ &  & \\ \hline
$3_{4,4_5}$ &  &  & $3$ &  & $3$ &  & $1$ & $2$ &  &  &  & $2$ &  &  &  &  &  & \\ \hline
$3_{4,4_5,2}$ &  &  &  & $3$ &  & $3$ & $2$ & $1$ &  &  &  &  &  &  &  &  &  & \\ \hline
$3_{1,2}$ & $3$ &  &  &  &  &  &  &  & $1$ & $3$ & $3$ &  &  &  &  &  & $3$ & \\ \hline
$3_{1,4_5,2}$ &  &  & $3$ &  &  &  &  &  & $3$ & $1$ &  & $3$ &  &  &  &  &  & \\ \hline
$3_{1,4,2}$ &  &  &  &  & $2$ &  &  &  & $3$ &  & $1$ & $3$ &  &  &  &  &  & $3$\\ \hline
$3_{1,4,4_5,2}$ &  &  &  &  &  &  & $2$ &  &  & $3$ & $3$ & $1$ &  &  &  &  &  & \\ \hline
$3_{6,4_5}$ & $3$ &  &  &  &  &  &  &  &  &  &  &  & $1$ & $2$ & $3$ &  & $3$ & \\ \hline
$3_{6,4_5,2}$ &  & $3$ &  &  &  &  &  &  &  &  &  &  & $2$ & $1$ &  & $3$ &  & \\ \hline
$3_{6,4,4_5}$ &  &  &  &  & $3$ &  &  &  &  &  &  &  & $3$ &  & $1$ & $2$ &  & $3$\\ \hline
$3_{6,4,4_5,2}$ &  &  &  &  &  & $3$ &  &  &  &  &  &  &  & $3$ & $2$ & $1$ &  & \\ \hline
$3_{6,1,4_5,2}$ &  &  &  &  &  &  &  &  & $3$ &  &  &  & $3$ &  &  &  & $1$ & $3$\\ \hline
$3_{6,1,4,4_5,2}$ &  &  &  &  &  &  &  &  &  &  & $3$ &  &  &  & $3$ &  & $3$ & $1$\\ \hline
\end{tabular} \hspace{20pt}
\begin{tabular}{|c||c|c|c|c|}
 \hhline{|=#=|=|=|=|}
$7$ & $1$ & $2$ & $3$ & \\ \hline
$7_1$ & $2$ & $1$ &  & $3$\\ \hline
$7_{6}$ & $3$ &  & $1$ & $2$\\ \hline
$7_{6,1}$ &  & $3$ & $2$ & $1$\\ \hline
\end{tabular}
\caption{Type-$3$ and type-$7$ adjacency matrices of $P_6$.}\label{t6}
\end{table}

\begin{figure}
\includegraphics{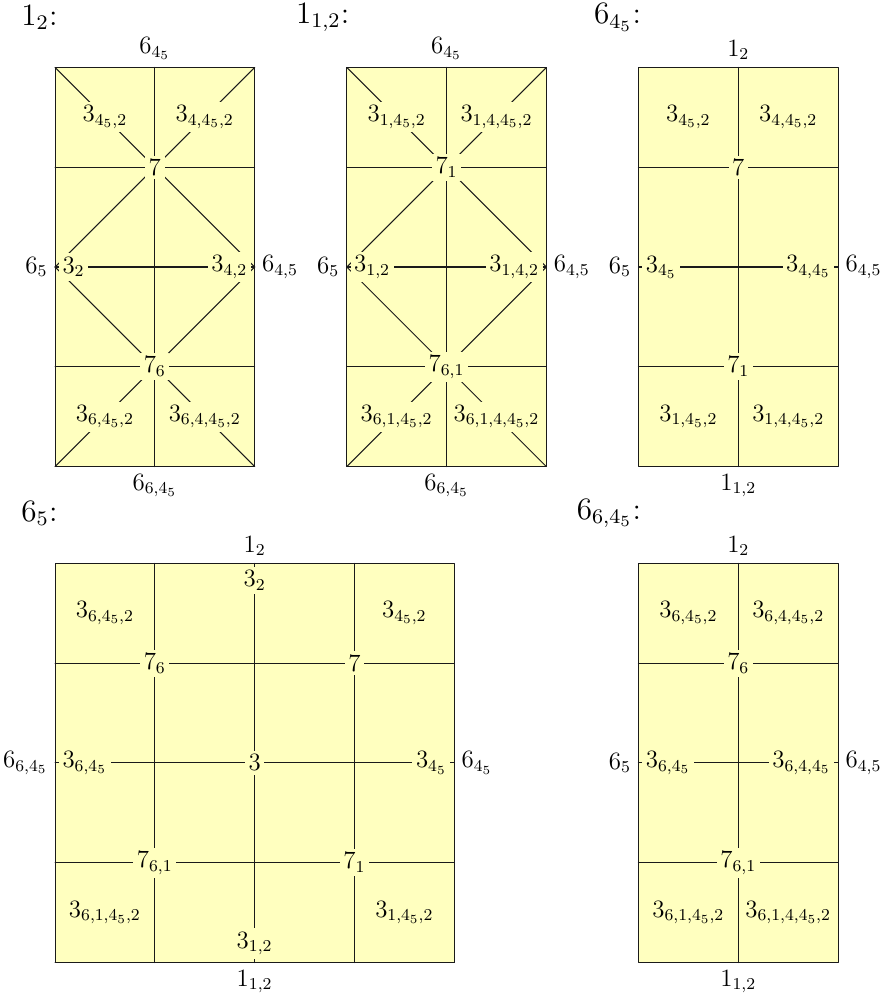}
    \caption{Some facets of $L_6$.}\label{f11}
\end{figure}

\begin{figure}
\includegraphics{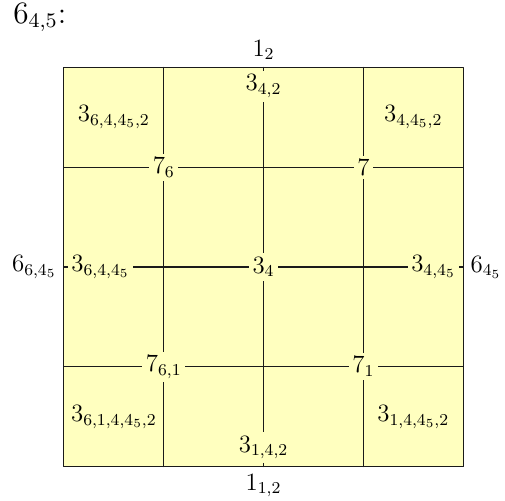}
    \caption{A facet of $L_6$.}\label{f11.1}
\end{figure}

\begin{table}
\scalebox{0.9}{
\begin{tabular}{|c||c|c|c|c|c|c|c|c|c|c|c|c|c|c|c|c|c|c|c|c|c|c|c|c|c|c|c|}
\hhline{|=#=|=|=|=|=|=|=|=|=|=|=|=|=|=|=|=|=|=|=|=|=|=|=|=|=|=|=|}
$3$ & $1$ & $2$ & $3$ &  & $3$ &  &  &  & $3$ &  &  &  & $3$ &  &  &  &  & & $3$ & & & & & & & & \\ \hline
$3_2$ & $2$ & $1$ &  & $3$ &  & $3$ &  &  &  &  &  &  &  & $3$ &  &  &  &  & & $3$ & & & & & & & \\ \hline
$3_{4_5}$ & $3$ &  & $1$ & $2$ &  &  & $3$ &  &  & $3$ &  &  &  &  &  &  &  & & & & $3$ & & & & & & \\ \hline
$3_{4_5,2}$ &  & $3$ & $2$ & $1$ &  &  &  & $3$ &  &  &  &  &  &  &  &  &  & & & & & $3$ & & & & & \\ \hline
$3_4$ & $3$ &  &  &  & $1$ & $2$ & $3$ &  &  &  & $2$ &  &  &  & $3$ &  &  & & & & & & & & & & \\ \hline
$3_{4,2}$ &  & $3$ &  &  & $2$ & $1$ &  & $3$ &  &  &  &  &  &  &  & $3$ &  & & & & & & & & & & \\ \hline
$3_{4,4_5}$ &  &  & $3$ &  & $3$ &  & $1$ & $2$ &  &  &  & $2$ &  &  &  &  &  & & & & & & & & & & \\ \hline
$3_{4,4_5,2}$ &  &  &  & $3$ &  & $3$ & $2$ & $1$ &  &  &  &  &  &  &  &  &  & & & & & & & & & & \\ \hline
$3_{1,2}$ & $3$ &  &  &  &  &  &  &  & $1$ & $3$ & $3$ &  &  &  &  &  & $3$ & & & & & & $3$ & & & & \\ \hline
$3_{1,4_5,2}$ &  &  & $3$ &  &  &  &  &  & $3$ & $1$ &  & $3$ &  &  &  &  &  & & & & & & & $3$ & & & \\ \hline
$3_{1,4,2}$ &  &  &  &  & $2$ &  &  &  & $3$ &  & $1$ & $3$ &  &  &  &  &  & $3$ & & & & & & & & &\\ \hline
$3_{1,4,4_5,2}$ &  &  &  &  &  &  & $2$ &  &  & $3$ & $3$ & $1$ &  &  &  &  &  & & & & & & & & & & \\ \hline
$3_{6,4_5}$ & $3$ &  &  &  &  &  &  &  &  &  &  &  & $1$ & $2$ & $3$ &  & $3$ & & & & & & & & $3$ & & \\ \hline
$3_{6,4_5,2}$ &  & $3$ &  &  &  &  &  &  &  &  &  &  & $2$ & $1$ &  & $3$ &  & & & & & & & & & $3$ & \\ \hline
$3_{6,4,4_5}$ &  &  &  &  & $3$ &  &  &  &  &  &  &  & $3$ &  & $1$ & $2$ &  & $3$ & & & & & & & & &\\ \hline
$3_{6,4,4_5,2}$ &  &  &  &  &  & $3$ &  &  &  &  &  &  &  & $3$ & $2$ & $1$ &  & & & & & & & & & & \\ \hline
$3_{6,1,4_5,2}$ &  &  &  &  &  &  &  &  & $3$ &  &  &  & $3$ &  &  &  & $1$ & $3$ & & & & & & & & & $3$\\ \hline
$3_{6,1,4,4_5,2}$ &  &  &  &  &  &  &  &  &  &  & $3$ &  &  &  & $3$ &  & $3$ & $1$ & & & & & & & & &\\ \hline
$3_{6_5,4}$ & $3$ & & & & & & & & & & & & & & & & & & $1$ & $2$ & $3$ & & $2$ & & $3$ & & \\ \hline
$3_{6_5,4,2}$ & & $3$ & & & & & & & & & & & & & & & & & $2$ & $1$ & & $3$ & & & & $3$ & \\ \hline
$3_{6_5,4,4_5}$ & & & $3$ & & & & & & & & & & & & & & & & $3$ & & $1$ & $2$ & & $2$ & & & \\ \hline
$3_{6_5,4,4_5,2}$ & & & & $3$ & & & & & & & & & & & & & & & & $3$ & $2$ & $1$ & & & & & \\ \hline
$3_{6_5,1,4,2}$ & & & & & & & & & $3$ & & & & & & & & & & $2$ & & & & $1$ & $3$ & & & $3$ \\ \hline
$3_{6_5,1,4,4_5,2}$ & & & & & & & & & & $3$ & & & & & & & & & & & $2$ & & $3$ & $1$ & & & \\ \hline
$3_{6_5,6,4,4_5}$ & & & & & & & & & & & & & $3$ & & & & & & $3$ & & & & & & $1$ & $2$ & $3$ \\ \hline
$3_{6_5,6,4,4_5,2}$ & & & & & & & & & & & & & & $3$ & & & & & & $3$ & & & & & $2$ & $1$ & \\ \hline
$3_{6_5,6,1,4,4_5,2}$ & & & & & & & & & & & & & & & & & $3$ & & & & & & $3$ & & $3$ & & $1$ \\ \hline
\end{tabular}
}

\vspace{20pt}

\begin{tabular}{|c||c|c|c|c|c|c|c|c|}
 \hhline{|=#=|=|=|=|=|=|=|=|}
$7$ & $1$ & $2$ & $3$ &  & $3$ &  &  & \\ \hline
$7_1$ & $2$ & $1$ &  & $3$ &  & $3$ &  & \\ \hline
$7_{6}$ & $3$ &  & $1$ & $2$ &  &  & $3$ & \\ \hline
$7_{6,1}$ &  & $3$ & $2$ & $1$ &  &  &  & $3$\\ \hline
$7_{6_5}$ & $3$ &  &  &  & $1$ & $2$ & $3$ & \\ \hline
$7_{6_5,1}$ &  & $3$ &  &  & $2$ & $1$ &  & $3$\\ \hline
$7_{6_5,6}$ &  &  & $3$ &  & $3$ &  & $1$ & $2$\\ \hline
$7_{6_5,6,1}$ &  &  &  & $3$ &  & $3$ & $2$ & $1$\\ \hline
\end{tabular}
\caption{Type-$3$ and type-$7$ adjacency matrices of $P_7$.}\label{t7}
\end{table}

\begin{figure}
\includegraphics{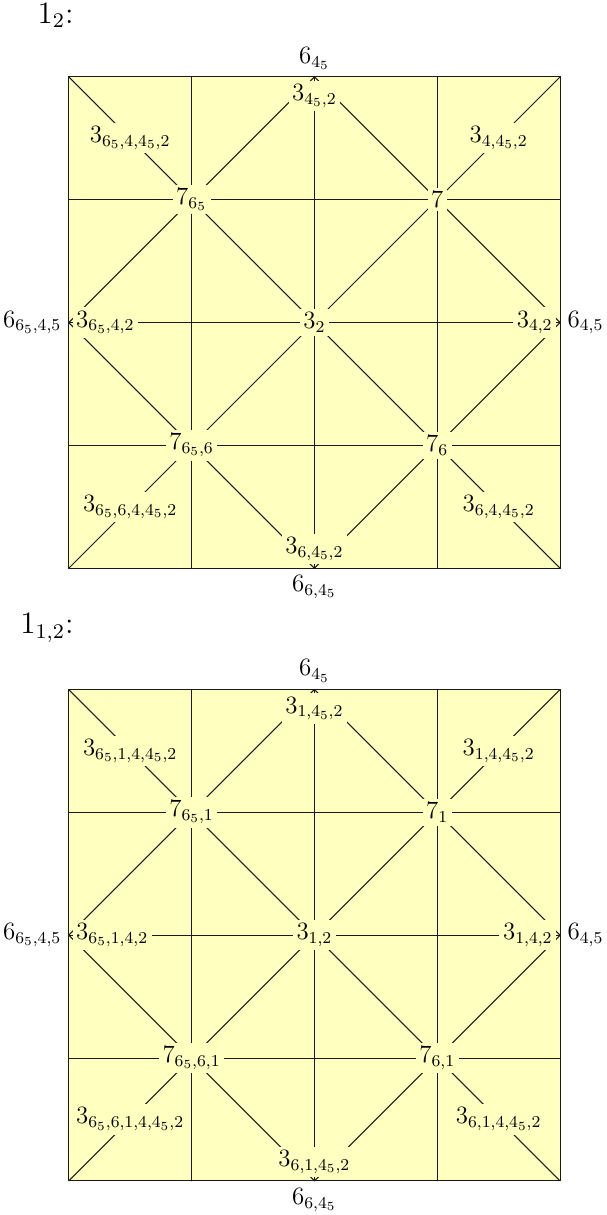}
    \caption{Some facets of $L_7$.}\label{f12}
\end{figure}

\begin{figure}
\includegraphics{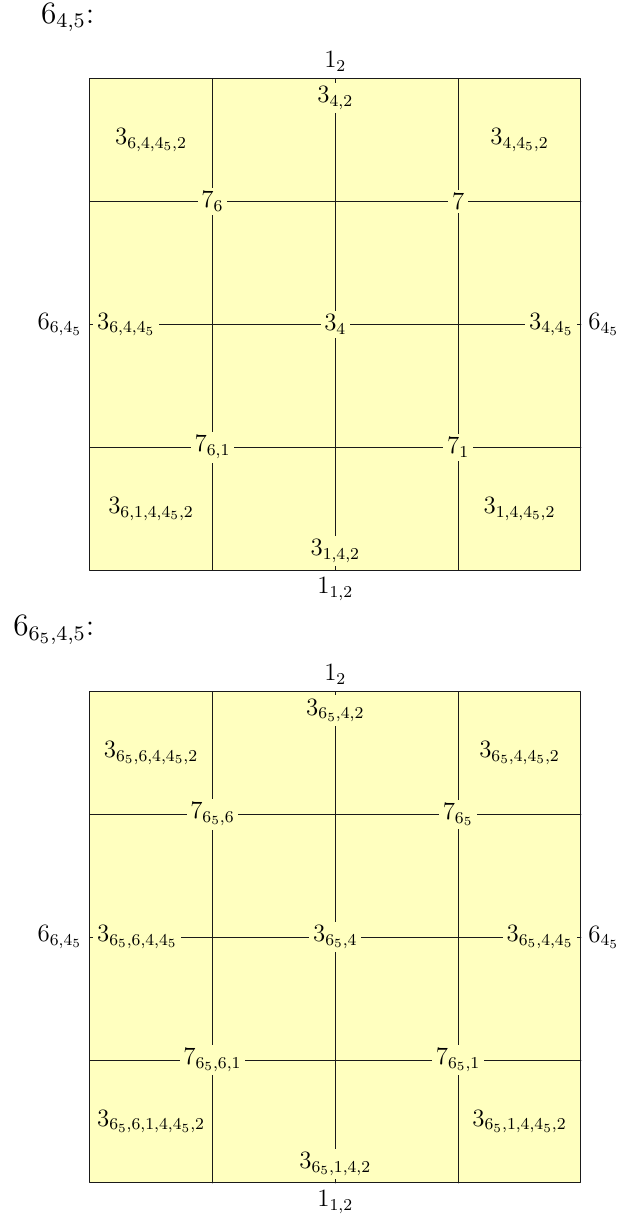}
    \caption{Some facets of $L_7$.}\label{f12.1}
\end{figure}

\begin{figure}
\includegraphics{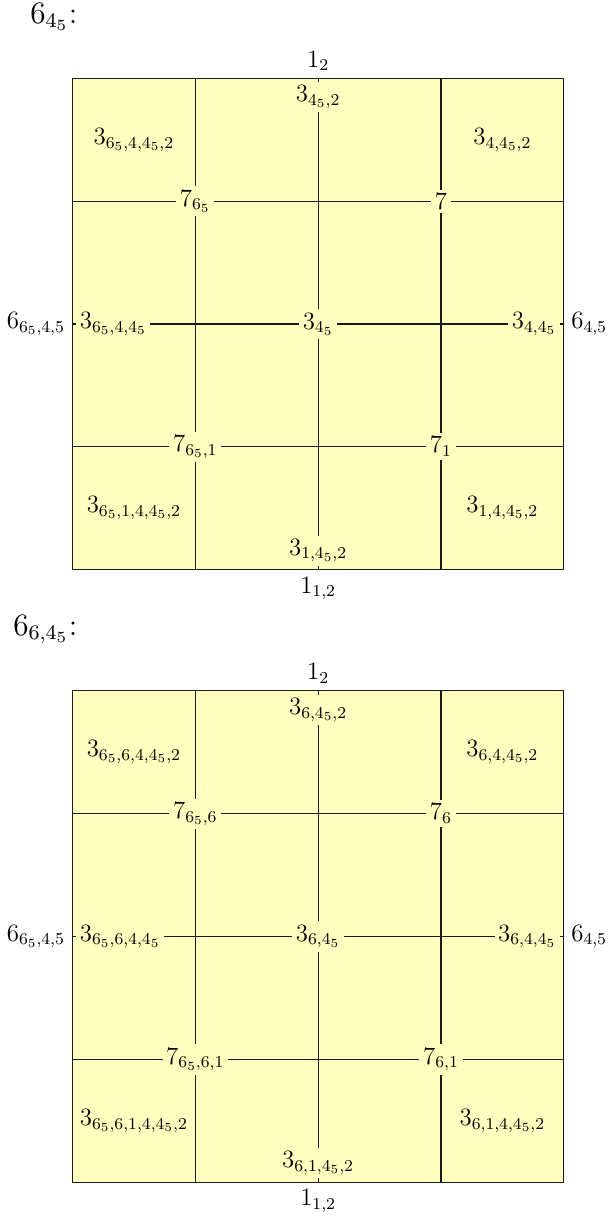}
    \caption{Some facets of $L_7$.}\label{f12.2}
\end{figure}

\begin{table}
\scalebox{0.58}{
\begin{tabular}{|c||c|c|c|c|c|c|c|c|c|c|c|c|c|c|c|c|c|c|c|c|c|c|c|c|c|c|c|c|c|c|c|c|c|c|c|c|c|c|c|c|c|c|c|c|c|}
\hhline{|=#=|=|=|=|=|=|=|=|=|=|=|=|=|=|=|=|=|=|=|=|=|=|=|=|=|=|=|=|=|=|=|=|=|=|=|=|=|=|=|=|=|=|=|=|=|}
$3$ & $1$ & $2$ & $3$ &  & $3$ &  &  &  & $3$ &  &  &  & $3$ &  &  &  &  & & $3$ & & & & & & & & & & & & & & & & & & & & & & & & & & \\ \hline
$3_2$ & $2$ & $1$ &  & $3$ &  & $3$ &  &  &  &  &  &  &  & $3$ &  &  &  &  & & $3$ & & & & & & & & $2$ & & & & & & & & & & & & & & & & & \\ \hline
$3_{4_5}$ & $3$ &  & $1$ & $2$ &  &  & $3$ &  &  & $3$ &  &  &  &  &  &  &  & & & & $3$ & & & & & & & & & & & & & & & & & & & & & & & & \\ \hline
$3_{4_5,2}$ &  & $3$ & $2$ & $1$ &  &  &  & $3$ &  &  &  &  &  &  &  &  &  & & & & & $3$ & & & & & & & $2$ & & & & & & & & & & & & & & & & \\ \hline
$3_4$ & $3$ &  &  &  & $1$ & $2$ & $3$ &  &  &  & $2$ &  &  &  & $3$ &  &  & & & & & & & & & & & & & & & & & & & & & & & & & & & & \\ \hline
$3_{4,2}$ &  & $3$ &  &  & $2$ & $1$ &  & $3$ &  &  &  &  &  &  &  & $3$ &  & & & & & & & & & & & & & $2$ & & & & & & & & & & & & & & & \\ \hline
$3_{4,4_5}$ &  &  & $3$ &  & $3$ &  & $1$ & $2$ &  &  &  & $2$ &  &  &  &  &  & & & & & & & & & & & & & & & & & & & & & & & & & & & & \\ \hline
$3_{4,4_5,2}$ &  &  &  & $3$ &  & $3$ & $2$ & $1$ &  &  &  &  &  &  &  &  &  & & & & & & & & & & & & & & $2$ & & & & & & & & & & & & & & \\ \hline
$3_{1,2}$ & $3$ &  &  &  &  &  &  &  & $1$ & $3$ & $3$ &  &  &  &  &  & $3$ & & & & & & $3$ & & & & & & & & & & & & & & & & & & & & & & \\ \hline
$3_{1,4_5,2}$ &  &  & $3$ &  &  &  &  &  & $3$ & $1$ &  & $3$ &  &  &  &  &  & & & & & & & $3$ & & & & & & & & & & & & & & & & & & & & & \\ \hline
$3_{1,4,2}$ &  &  &  &  & $2$ &  &  &  & $3$ &  & $1$ & $3$ &  &  &  &  &  & $3$ & & & & & & & & & & & & & & & & & & & & & & & & & & & \\ \hline
$3_{1,4,4_5,2}$ &  &  &  &  &  &  & $2$ &  &  & $3$ & $3$ & $1$ &  &  &  &  &  & & & & & & & & & & & & & & & & & & & & & & & & & & & & \\ \hline
$3_{6,4_5}$ & $3$ &  &  &  &  &  &  &  &  &  &  &  & $1$ & $2$ & $3$ &  & $3$ & & & & & & & & $3$ & & & & & & & & & & & & & & & & & & & & \\ \hline
$3_{6,4_5,2}$ &  & $3$ &  &  &  &  &  &  &  &  &  &  & $2$ & $1$ &  & $3$ &  & & & & & & & & & $3$ & & & & & & & & & & $2$ & & & & & & & & & \\ \hline
$3_{6,4,4_5}$ &  &  &  &  & $3$ &  &  &  &  &  &  &  & $3$ &  & $1$ & $2$ &  & $3$ & & & & & & & & & & & & & & & & & & & & & & & & & & & \\ \hline
$3_{6,4,4_5,2}$ &  &  &  &  &  & $3$ &  &  &  &  &  &  &  & $3$ & $2$ & $1$ &  & & & & & & & & & & & & & & & & & & & & $2$ & & & & & & & & \\ \hline
$3_{6,1,4_5,2}$ &  &  &  &  &  &  &  &  & $3$ &  &  &  & $3$ &  &  &  & $1$ & $3$ & & & & & & & & & $3$ & & & & & & & & & & & & & & & & & & \\ \hline
$3_{6,1,4,4_5,2}$ &  &  &  &  &  &  &  &  &  &  & $3$ &  &  &  & $3$ &  & $3$ & $1$ & & & & & & & & & & & & & & & & & & & & & & & & & & & \\ \hline
$3_{6_5,4}$ & $3$ & & & & & & & & & & & & & & & & & & $1$ & $2$ & $3$ & & $2$ & & $3$ & & & & & & & & & & & & & & & & & & & & \\ \hline
$3_{6_5,4,2}$ & & $3$ & & & & & & & & & & & & & & & & & $2$ & $1$ & & $3$ & & & & $3$ & & & & & & & & & & & & & & & & $2$ & & & \\ \hline
$3_{6_5,4,4_5}$ & & & $3$ & & & & & & & & & & & & & & & & $3$ & & $1$ & $2$ & & $2$ & & & & & & & & & & & & & & & & & & & & & \\ \hline
$3_{6_5,4,4_5,2}$ & & & & $3$ & & & & & & & & & & & & & & & & $3$ & $2$ & $1$ & & & & & & & & & & & & & & & & & & & $2$ & & & & \\ \hline
$3_{6_5,1,4,2}$ & & & & & & & & & $3$ & & & & & & & & & & $2$ & & & & $1$ & $3$ & & & $3$ & & & & & & & & & & & & & & & & & & \\ \hline
$3_{6_5,1,4,4_5,2}$ & & & & & & & & & & $3$ & & & & & & & & & & & $2$ & & $3$ & $1$ & & & & & & & & & & & & & & & & & & & & & \\ \hline
$3_{6_5,6,4,4_5}$ & & & & & & & & & & & & & $3$ & & & & & & $3$ & & & & & & $1$ & $2$ & $3$ & & & & & & & & & & & & & & & & & & \\ \hline
$3_{6_5,6,4,4_5,2}$ & & & & & & & & & & & & & & $3$ & & & & & & $3$ & & & & & $2$ & $1$ & & & & & & & & & & & & & & & & & & $2$ & \\ \hline
$3_{6_5,6,1,4,4_5,2}$ & & & & & & & & & & & & & & & & & $3$ & & & & & & $3$ & & $3$ & & $1$ & & & & & & & & & & & & & & & & & & \\ \hline
$3_{1_2}$ & & $2$ & & & & & & & & & & & & & & & & & & & & & & & & & & $1$ & $3$ & $3$ & & $3$ & & & & $3$ & & & & $3$ & & & & & \\ \hline
$3_{1_2,4_5}$ & & & & $2$ & & & & & & & & & & & & & & & & & & & & & & & & $3$ & $1$ & & $3$ & & $3$ & & & & & & & & $3$ & & & & \\ \hline
$3_{1_2,4}$ & & & & & & $2$ & & & & & & & & & & & & & & & & & & & & & & $3$ & & $1$ & $3$ & & & $2$ & & & $3$ & & & & & & & & \\ \hline
$3_{1_2,4,4_5}$ & & & & & & & & $2$ & & & & & & & & & & & & & & & & & & & & & $3$ & $3$ & $1$ & & & & $2$ & & & & & & & & & & \\ \hline
$3_{1_2,1,2}$ & & & & & & & & & & & & & & & & & & & & & & & & & & & & $3$ & & & & $1$ & $3$ & $3$ & & & & $3$ & & & & $3$ & & & \\ \hline
$3_{1_2,1,4_5,2}$ & & & & & & & & & & & & & & & & & & & & & & & & & & & & & $3$ & & & $3$ & $1$ & & $3$ & & & & & & & & $3$ & & \\ \hline
$3_{1_2,1,4,2}$ & & & & & & & & & & & & & & & & & & & & & & & & & & & & & & $2$ & & $3$ & & $1$ & $3$ & & & & $3$ & & & & & & \\ \hline
$3_{1_2,1,4,4_5,2}$ & & & & & & & & & & & & & & & & & & & & & & & & & & & & & & & $2$ & & $3$ & $3$ & $1$ & & & & & & & & & & \\ \hline
$3_{1_2,6,4_5}$ & & & & & & & & & & & & & & $2$ & & & & & & & & & & & & & & $3$ & & & & & & & & $1$ & $3$ & $3$ & & & & & & $3$ & \\ \hline
$3_{1_2,6,4,4_5}$ & & & & & & & & & & & & & & & & $2$ & & & & & & & & & & & & & & $3$ & & & & & & $3$ & $1$ & & $3$ & & & & & & \\ \hline
$3_{1_2,6,1,4_5,2}$ & & & & & & & & & & & & & & & & & & & & & & & & & & & & & & & & $3$ & & & & $3$ & & $1$ & $3$ & & & & & & $3$ \\ \hline
$3_{1_2,6,1,4,4_5,2}$ & & & & & & & & & & & & & & & & & & & & & & & & & & & & & & & & & & $3$ & & & $3$ & $3$ & $1$ & & & & & & \\ \hline
$3_{1_2,6_5,4}$ & & & & & & & & & & & & & & & & & & & & $2$ & & & & & & & & $3$ & & & & & & & & & & & & $1$ & $3$ & $2$ & & $3$ & \\ \hline
$3_{1_2,6_5,4,4_5}$ & & & & & & & & & & & & & & & & & & & & & & $2$ & & & & & & & $3$ & & & & & & & & & & & $3$ & $1$ & & $2$ & & \\ \hline
$3_{1_2,6_5,1,4,2}$ & & & & & & & & & & & & & & & & & & & & & & & & & & & & & & & & $3$ & & & & & & & & $2$ & & $1$ & $3$ & & $3$ \\ \hline
$3_{1_2,6_5,1,4,4_5,2}$ & & & & & & & & & & & & & & & & & & & & & & & & & & & & & & & & & $3$ & & & & & & & & $2$ & $3$ & $1$ & & \\ \hline
$3_{1_2,6_5,6,4,4_5}$ & & & & & & & & & & & & & & & & & & & & & & & & & & $2$ & & & & & & & & & & $3$ & & & & $3$ & & & & $1$ & $3$ \\ \hline
$3_{1_2,6_5,6,1,4,4_5,2}$ & & & & & & & & & & & & & & & & & & & & & & & & & & & & & & & & & & & & & & $3$ & & & & $3$ & & $3$ & $1$ \\ \hline
\end{tabular}
}

\vspace{20pt}

\begin{tabular}{|c||c|c|c|c|c|c|c|c|c|c|c|c|c|c|c|c|}
\hhline{|=#=|=|=|=|=|=|=|=|=|=|=|=|=|=|=|=|}
$7$ & $1$ & $2$ & $3$ &  & $3$ &  &  & & $2$ & & & & & & & \\ \hline
$7_1$ & $2$ & $1$ &  & $3$ &  & $3$ &  & & & & & & & & & \\ \hline
$7_{6}$ & $3$ &  & $1$ & $2$ &  &  & $3$ & & & & $2$ & & & & & \\ \hline
$7_{6,1}$ &  & $3$ & $2$ & $1$ &  &  &  & $3$ & & & & & & & & \\ \hline
$7_{6_5}$ & $3$ &  &  &  & $1$ & $2$ & $3$ & & & & & & $2$ & & & \\ \hline
$7_{6_5,1}$ &  & $3$ &  &  & $2$ & $1$ &  & $3$ & & & & & & & & \\ \hline
$7_{6_5,6}$ &  &  & $3$ &  & $3$ &  & $1$ & $2$ & & & & & & & $2$ & \\ \hline
$7_{6_5,6,1}$ &  &  &  & $3$ &  & $3$ & $2$ & $1$ & & & & & & & & \\ \hline
$7_{1_2}$ & $2$ & & & & & & & & $1$ & $2$ & $3$ & & $3$ & & & \\ \hline
$7_{1_2,1}$ & & & & & & & & & $2$ & $1$ & & $3$ & & $3$ & & \\ \hline
$7_{1_2,6}$ & & & $2$ & & & & & & $3$ & & $1$ & $2$ & & & $3$ & \\ \hline
$7_{1_2,6,1}$ & & & & & & & & & & $3$ & $2$ & $1$ & & & & $3$ \\ \hline
$7_{1_2,6_5}$ & & & & & $2$ & & & & $3$ & & & & $1$ & $2$ & $3$ & \\ \hline
$7_{1_2,6_5,1}$ & & & & & & & & & & $3$ & & & $2$ & $1$ & & $3$ \\ \hline
$7_{1_2,6_5,6}$ & & & & & & & $2$ & & & & $3$ & & $3$ & & $1$ & $2$ \\ \hline
$7_{1_2,6_5,6,1}$ & & & & & & & & & & & & $3$ & & $3$ & $2$ & $1$ \\ \hline
\end{tabular}
\caption{Type-$3$ and type-$7$ adjacency matrices of $P_8$.}\label{t8}
\end{table}

\begin{figure}
\includegraphics{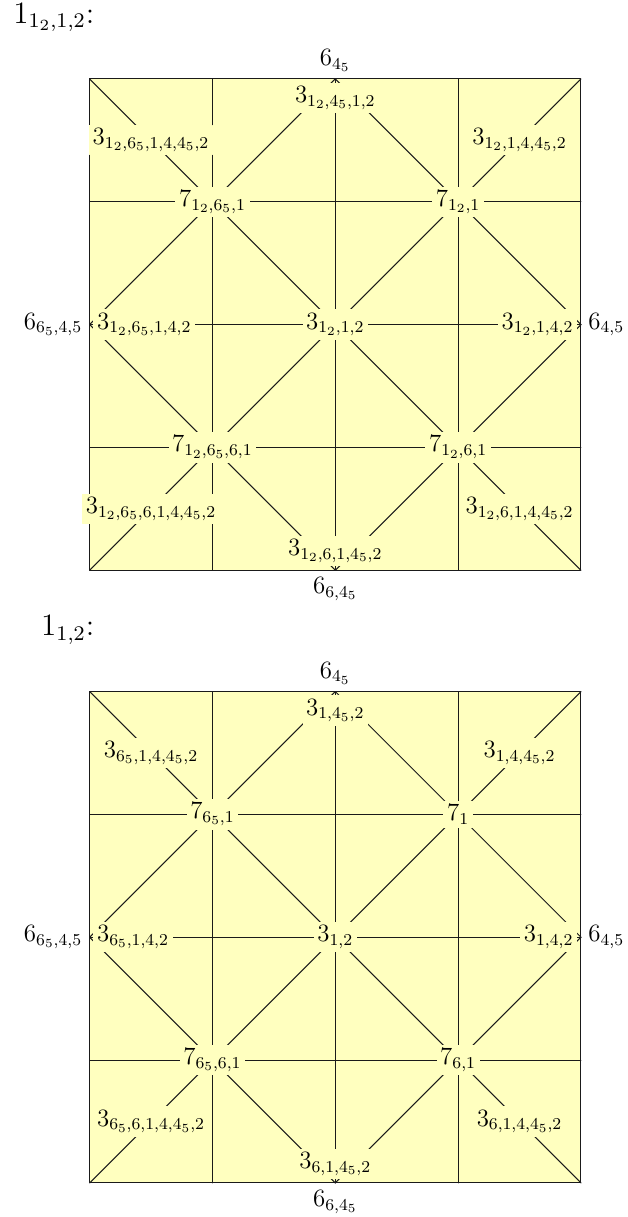}
    \caption{Some facets of $L_8$.}\label{f13}
\end{figure}

\begin{figure}
\includegraphics{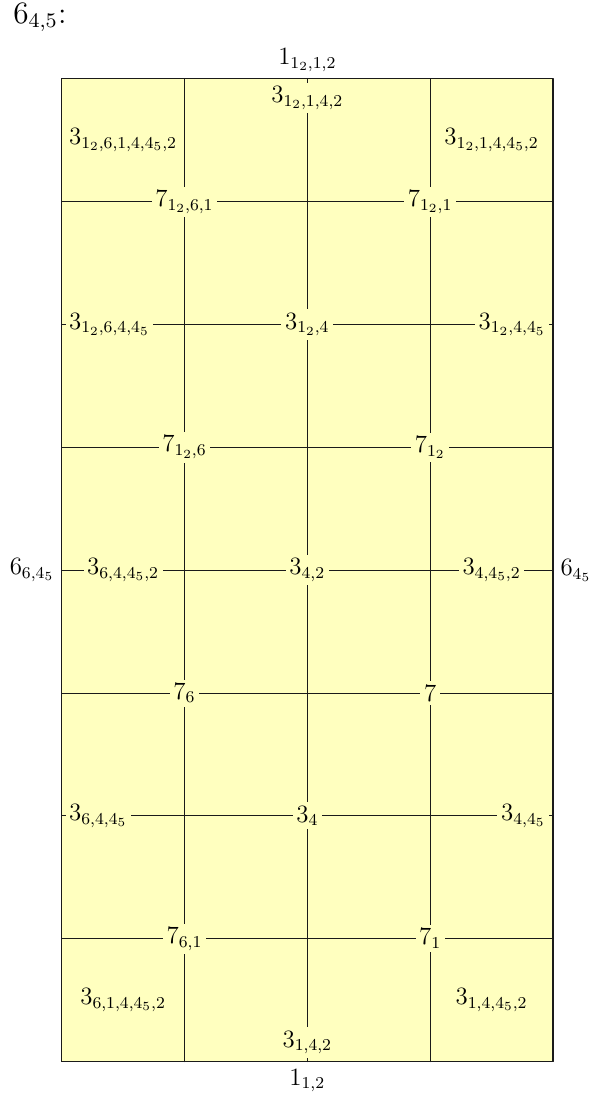}
    \caption{A facet of $L_8$.}\label{f13.1}
\end{figure}

\begin{figure}
\includegraphics{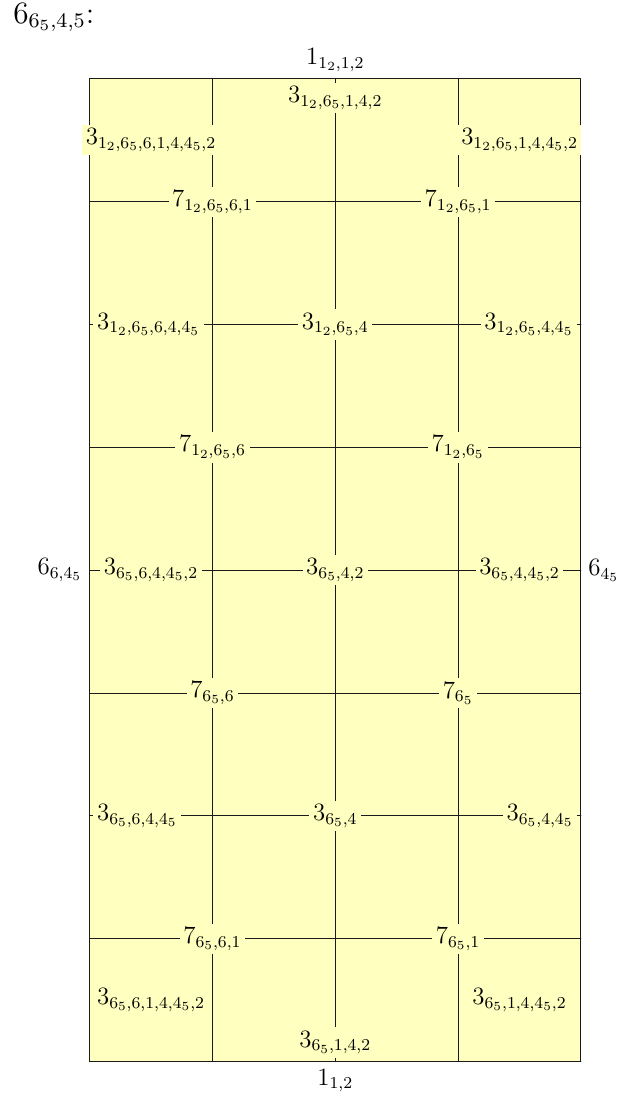}
    \caption{A facet of $L_8$.}\label{f13.2}
\end{figure}

\begin{figure}
\includegraphics{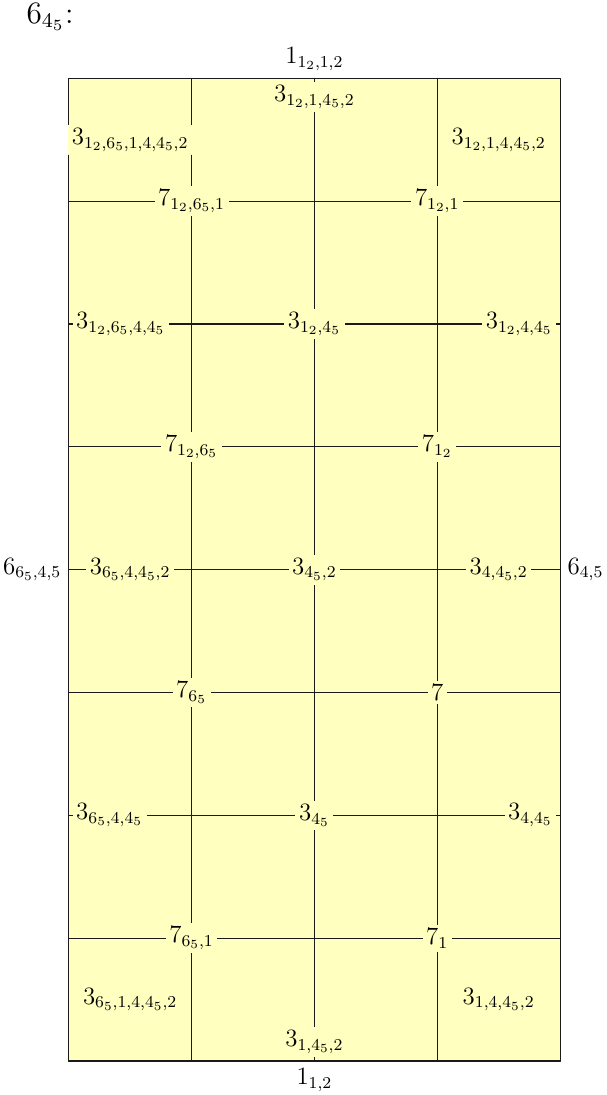}
    \caption{A facet of $L_8$.}\label{f13.3}
\end{figure}

\begin{figure}
\includegraphics{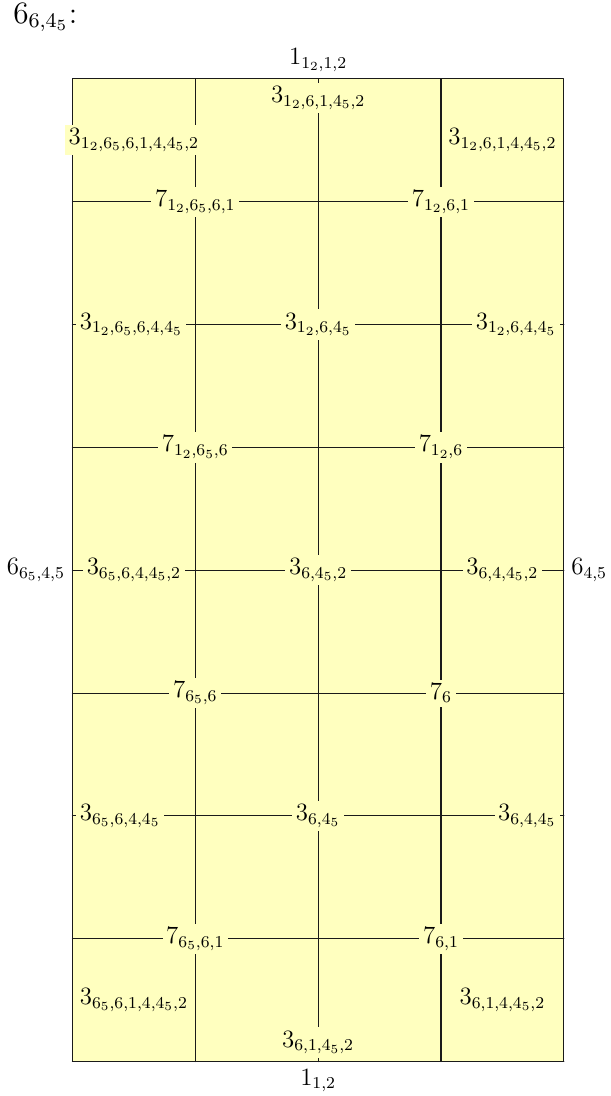}
    \caption{A facet of $L_8$.}\label{f13.4}
\end{figure}

\clearpage

\bibliographystyle{alpha}
\bibliography{bibliografia}

\end{document}